\documentclass[11pt, a4paper]{amsart}
\usepackage[T1]{fontenc}
\usepackage{amsmath,amsthm}
\usepackage{amssymb}
\usepackage{mathabx}
\usepackage{color}
\usepackage{dsfont}

\newtheorem{prop}{Proposition}[section]
\newtheorem{teo}{Theorem}[section]
\newtheorem{lema}{Lemma}[section]
\newtheorem{coro}{Corollary}[section]

\theoremstyle{definition}
\newtheorem{rem}{Remark}[section]

\def\ep{\varepsilon}

\def\a{\alpha}

\def\R{\mathbb R}

\def\K{{\mathcal K}}

\def\T{{\mathcal T}}

\begin{document}
\title[A fully nonlocal semilinear problem]{A semilinear problem associated to the \\ space-time fractional heat equation in $\R^N$}

\author[C. Cort\'{a}zar,  F. Quir\'{o}s \and N. Wolanski]{Carmen Cort\'{a}zar,  Fernando  Quir\'{o}s, \and Noem\'{\i} Wolanski}

\address{Carmen Cort\'{a}zar\hfill\break\indent
	Departamento de Matem\'{a}tica, Pontificia Universidad Cat\'{o}lica de Chile
    \hfill\break\indent Santiago, Chile.} \email{{\tt ccortaza@mat.puc.cl} }

\address{Fernando Quir\'{o}s\hfill\break
    \indent Departamento  de Matem\'{a}ticas, Universidad Aut\'{o}noma de Madrid, \hfill\break
    \indent 28049-Madrid, Spain, \hfill\break
    \indent and Instituto de Ciencias Matem\'aticas ICMAT (CSIC-UAM-UCM-UC3M), \hfill\break
    \indent 28049-Madrid, Spain.} \email{{\tt fernando.quiros@uam.es} }

\address{Noem\'{\i} Wolanski \hfill\break\indent
	IMAS-UBA-CONICET,
    \hfill\break\indent Ciudad Universitaria, Pab. I,
    \hfill\break\indent (1428) Buenos Aires, Argentina.} \email{{\tt wolanski@dm.uba.ar} }



\date{}

\begin{abstract} We study the fully nonlocal semilinear equation $\partial_t^\a u+(-\Delta)^\beta u=|u|^{p-1}u$, $p\ge1$, where $\partial_t^\a$ stands for the Caputo derivative of order $\alpha\in (0,1)$ and $(-\Delta)^\beta$, $\beta\in(0,1]$, is the usual $\beta$ power of the Laplacian. We prescribe an initial datum in $L^q(\mathbb{R}^N)$.

We give conditions ensuring the existence and uniqueness of a solution living in $L^q(\R^N)$ up to a maximal existence time $T$ that may be finite or infinite. If~$T$ is finite, the $L^q$ norm of the solution becomes unbounded as time approaches $T$, and $u$ is said to blow up in $L^q$. Otherwise, the solution is global in time.

For the case of nonnegative and nontrivial solutions, we give conditions on the initial datum that ensure either blow-up or global existence. It turns out that every nonnegative nontrivial solution in $L^q$ blows up in finite time if $1<p<p_f:=1+\frac{2\beta}N$ whereas if $p\ge p_f$ there are both solutions that blow up and global ones. The critical exponent $p_f$, which does not depend on $\alpha$, coincides with the Fujita exponent for the case $\alpha=1$, in which the time derivative is the standard (local) one. In contrast to the case $\alpha=1$, when $\alpha\in(0,1)$ the critical exponent $p=p_f$ falls within the situation in which global existence may occur. Our weakest condition for global existence and our condition for blow-up are both related to the size of the mean value of the initial datum in large balls.
\end{abstract}

\maketitle

\section{Introduction}
 \setcounter{equation}{0}

\subsection{Aim}

We study the \emph{fully nonlocal semilinear} problem
\begin{equation}\label{eq-problem}
 	\partial_t^\a u+(-\Delta)^\beta u=|u|^{p-1}u\quad\mbox{in }\R^N\times(0,T),\qquad
 	u(\cdot,0)=u_0\quad\mbox{in }\R^N,
\end{equation}
where $p\in[1,\infty)$. Here $\partial_t^\alpha$, $\alpha\in(0,1)$, denotes the so-called Caputo $\alpha$-derivative, defined for smooth functions by
$$
\displaystyle\partial_t^\alpha u(x,t)=\frac1{ \Gamma (1-\alpha)}\,\partial_t\int_0^t\frac{u(x,s)- u(x,0)}{(t-s)^{\alpha}}\,{\rm d}s,
$$
and $(-\Delta)^\beta$,  $\beta\in(0,1]$, is the usual $\beta$ power of the Laplacian, defined for smooth functions when $\beta\in(0,1)$ by
$$
\displaystyle (-\Delta)^\beta v(x)=c_{N,\beta}\operatorname{P.V.}\int_{\R^N}\frac{v(x)-v(y)}{|x-y|^{N+2\beta}}\,{\rm d}y.
$$
The positive normalization constant $c_{N,\beta}$ is chosen so that $(-\Delta)^{\beta}=\mathcal{F}^{-1}(|\cdot|^{2\beta}\mathcal{F})$, where $\mathcal{F}$ denotes Fourier transform. Problems like~\eqref{eq-problem}, nonlocal both in space and time, are useful to model situations with long-range interactions and memory effects; see for instance~\cite{Metzler-Klafter-2000}.

If a solution to~\eqref{eq-problem} is smooth enough, it is  given by Duhamel's type formula
\begin{equation}\label{eq-formula}
 		u(x,t)=\int_{\R^N} u_0(y)Z(x-y,t)\,{\rm d}y+\int_0^t\int_{\R^N}|u(y,s)|^{p-1}u(y,s)Y(x-y,t-s)\,{\rm d}y{\rm d}s,
\end{equation}
where $Z$ is the fundamental solution of the equation and
$Y={}^{\rm R}\partial_t^{1-\alpha}Z$; see~\cite{Kemppainen-Siljander-Zacher-2017}. Here, ${}^{\rm R}\partial_t^{\alpha}$ stands for the so-called Riemann-Liouville $\alpha$-derivative, defined by
$$
{}^{\rm R}\partial_t^\alpha u(x,t)=\frac1{ \Gamma (1-\alpha)}\,\partial_t\int_0^t\frac{u(x,s)}{(t-s)^{\alpha}}\,{\rm d}s.
$$
Throughout the article we consider $u_0\in L^q(\mathbb{R}^N)$, $q\in[1,\infty]$, and look for $L^q$-\emph{solutions} of~\eqref{eq-problem} with initial datum $u_0$ and existence time $T$, that is, functions $u\in L^\infty_{\rm loc}([0,T);L^q(\mathbb{R}^N))$ satisfying~\eqref{eq-formula} in $L^q(\mathbb{R}^N)$ for almost all $t\in(0,T)$, and such that
\begin{equation}
\label{eq:initial.datum.Lq}
\operatornamewithlimits{ess\,lim}_{t\to0^+}\|u(\cdot,t)-u_0\|_{L^q(\mathbb{R}^N)}=0.
\end{equation}

\begin{rem} The Caputo and the Riemann-Liouville $\alpha$-derivatives are related by
$$
\partial_t^\alpha u(\cdot,t)={}^{\rm R}\partial_t^\alpha (u(\cdot,t)-u(\cdot,0)).
$$
Thus, both derivatives coincide for functions with zero initial datum.
\end{rem}

Our main goal is to give conditions under which problem~\eqref{eq-problem} admits either  $L^q$-\emph{global} solutions, that is,  $L^q$-solutions with existence time $T=\infty$, or  $L^q$-solutions with \emph{blow-up} in finite time, that is, $L^q$-solutions such that $\limsup_{t\nearrow T}\|u(\cdot,t)\|_{L^q(\mathbb{R}^N)}=\infty$ for some $T\in(0,\infty)$.

In order to prove our results we use thoroughly the properties of $Z$ and $Y$. These kernels are not smooth in general and have a qualitative change of behavior across the critical dimensions $N=2\beta$ and $N=4\beta$. For the sake of simplicity, throughout this article we will only deal with the more involved case $N>4\beta$, for which the kernels are more singular. The other cases can be treated similarly. We recall the estimates on $Z$ and $Y$ required for our proofs, due to different authors, in Section~\ref{sect:estimates}.

\noindent\emph{Historical note. } The nonlocal derivative $\partial_t^\alpha$ has been introduced independently by many authors; see for instance \cite{Caputo-1967,Dzherbashyan-Nersesian-1968,Gerasimov-1948,Gross-1947,Rabotnov-1966}. It can be found even in a very old paper by Liouville~\cite{Liouville-1832}. However, we follow the common convention of naming it after Caputo only.

\subsection{Main results}

The first part of the paper is devoted to basic issues such as existence, uniqueness, prolongability and continuity. We assume always that $p\in[1,\infty)$.

We first  consider, in~Section~\ref{sect-local existence}, values of $q$ in the \emph{good range},
\begin{equation}
\label{eq:good.range}
\tag{GR}
q\in(\ell,\infty]\text{ with } q\ge p,\quad\text{where }
\ell:=\frac{N}{2\beta}(p-1).
\end{equation}
Notice that $q=\infty$ is always in the good range, while $q=1$ only belongs to it if $p=1$. If $q$ satisfies~\eqref{eq:good.range} we are able to prove the existence of a unique \emph{maximal} $L^q$-solution, defined up to a maximal positive existence time $T$, for any initial datum in~$L^q(\mathbb{R}^N)$. Moreover, $u\in C([0,T); L^q(\mathbb{R}^N))$. If $T$ is finite, the solution blows up. When $p=1$, $L^q$-solutions are always global in time, that is, $T=\infty$.

If $q\in[1,p)$, difficulties arise in the interpretation of~\eqref{eq-formula}, since $|u(\cdot,t)|^{p-1}u(\cdot,t)$ need not \emph{a priori} be in $L^q(\mathbb{R}^N)$. On the other hand, if $q\in[1,\ell]$, the problem is not expected to be well posed in $L^q(\mathbb{R}^N)$, even if $q\ge p$. In particular, we expect to have initial data in $L^q(\mathbb{R}^N)$ for which no $L^q$-solution exists; see Paragraph~\ref{subsubsect:instantaneous} below.  To face these difficulties, to guarantee the existence of a solution (that may not be unique; see Paragraph~\ref{subsubsect:nonuniqueness}) outside the good range we require further integrability to the initial data. Thus, in Section~\ref{sect:basic.theory.outside.GR}:
\begin{itemize}
\item If $q\in(\ell,p)$, we consider $u_0\in L^q(\mathbb{R}^N)\cap L^p(\mathbb{R}^N)$ and we show that the corresponding maximal $L^p$-solution is an $L^q$-solution with the same existence time.
\item If $q\in[1,\ell]$, we consider $u_0\in L^q(\mathbb{R}^N)\cap L^{\hat q}(\mathbb{R}^N)$ for some $\hat q$ in the good range and show that the corresponding maximal $L^{\hat q}$-solution is an $L^q$-solution with the same existence time.
\end{itemize}

In the second part of the paper we seek for conditions under which a nonnegative $L^q$-solution, $q\in[1,\infty]$, is either global in time or, on the contrary, blows up in finite time when $p\in(1,\infty)$. We find a critical exponent, $p_f:=1+\frac{2\beta}N$, which, by analogy with the local case, we call the \emph{Fujita exponent}, such that every nonnegative and nontrivial $L^q$-solution blows up in finite time if $p\in(1,p_f)$ whereas for $p\in[p_f,\infty)$ there are both global and blowing up solutions. We remark that the value of the Fujita exponent does not depend on the fractional order $\alpha$.
	
We begin our analysis of the existence of a Fujita exponent in Section \ref{sect-global existence}, where we construct global in time $L^q$-solutions for all $p\in [p_f,\infty)$ and $q\in[1,\infty]$.
We start by considering values of $q$ in the \emph{distinguished range},
\begin{equation}\label{eq:distinguished.range}
\tag{DR}
  q\in(\ell,\ell p]\text{ with } q\ge p
\end{equation}
(we recall that we are assuming now that $p\ge p_f$, so that $\ell\ge1$).
Such values of $q$ are contained in the good range~\eqref{eq:good.range}. Therefore, given any $u_0\in L^q(\mathbb{R}^N)$ there is a maximal $L^q$-solution to~\eqref{eq-problem}. In order to ensure that this $L^q$-solution is global in time we require in addition that $u_0\in L^\ell(\mathbb{R}^N)$ and one of the two following conditions:
\begin{align}
\label{eq:l.norm.small}
&\|u_0\|_{L^\ell(\mathbb{R}^N)}\quad\text{small},\\
\label{eq:average.condition}
&u_0\ge 0, \quad\text{with }
\sup_{R>R_0} R^{\frac\alpha{p-1}-\frac {N\alpha}{2\beta }}\int_{|x|<R^{\frac\alpha{2\beta}}}u_0(x)\,{\rm d}x<\delta,
\end{align}
where $R_0=R_0(\|u_0\|_{L^q(\mathbb{R}^N)})>0$ decreases with $\|u_0\|_{L^q(\mathbb{R}^N)}$ and $\delta>0$ is independent of $u_0$. Observe that for every $R>0$
\[
R^{\frac\alpha{p-1}-\frac{N\alpha}{2\beta }}\int_{|x|<R^{\frac\alpha{2\beta}}}u_0(x)\,{\rm d}x\le \omega_N^{\frac{\ell-1}{\ell}}\|u_0\|_{L^\ell(\R^N)},
\]
where $\omega_N$ is the volume of the unit ball in $\mathbb{R}^N$, showing that the second condition is of a weaker nature than the first one, except for $p=p_f$, when both conditions coincide, since in this case $\ell=1$ and $\frac1{p-1}-\frac{N}{2\beta }=0$.

To construct global solutions outside the distinguished range, we have to require further integrability to the initial datum. In particular:
\begin{itemize}
\item If $q\in(\ell,p)$, we consider $u_0\in L^\ell(\mathbb{R}^N)\cap L^p(\mathbb{R}^N)$ satisfying either~\eqref{eq:l.norm.small} or \eqref{eq:average.condition}, and prove that the corresponding global $L^p$-solution is a global $L^q$-solution.
\item If $q\in[1,\ell]$, we consider $u_0\in L^q(\mathbb{R}^N)\cap L^{\hat q}(\mathbb{R}^N)$ for some $\hat q$ in the good range, and satisfying either~\eqref{eq:l.norm.small} or \eqref{eq:average.condition}, and prove that the corresponding global $L^{\hat q}$-solution is a global $L^q$-solution.
\item If $q\in(\ell p,\infty)$, we consider $u_0\in L^\ell(\mathbb{R}^N)\cap L^q(\mathbb{R}^N)$ satisfying either~\eqref{eq:l.norm.small} or \eqref{eq:average.condition}, and prove that the corresponding global $L^{\ell p}$-solution is a global $L^q$-solution.
\end{itemize}

Finally, in Section~\ref{sect-blow up} we prove that there is a constant $C$ depending only on $N,\alpha,\beta$ and $p$ such that if $u_0\in L^q(\mathbb{R}^N)$ is nonnegative and
\[
R^{\frac\alpha{p-1}-\frac{N\alpha}{2\beta }}\int_{|x|<R^{\frac\alpha{2\beta}}}u_0(x)\,{\rm d}x>C \quad\text{for some }R>0,
\]
then problem~\eqref{eq-problem} cannot have a global $L^q$-solution with $u_0$ as initial datum.  This shows that condition~\eqref{eq:average.condition} for global existence is \lq\lq sharp''. As a consecuence, on the one hand there are $L^q$-solutions with blow-up in finite time for every $p\in(1,\infty)$ and $q\in[1,\infty]$, and on the other hand, since $\frac\alpha{p-1}-\frac N{2\beta }>0$ if $p\in(1,p_f)$, every nonnegative and nontrivial $L^q$-solution to problem~\eqref{eq-problem} blows up in finite time in this range of values of $p$. 	
 	
\subsection{Precedents}

In the case of the classical heat equation, $\alpha=1$, $\beta=1$, Fujita proved in his classical paper \cite{Fujita-1966} that problem~\eqref{eq-problem} admits no global $L^\infty$-solution when $p\in(1,1+2/N)$, while it has both global $L^\infty$-solutions and solutions which blow up in $L^\infty$ in finite time if $p>1+2/N$.  When $p=1+2/N$, it was shown later in \cite{Ha} (for $N=1,2$), \cite{Kobayashi-Sirao-Tanaka-1977} and \cite{AW}, that all nontrivial solutions to problem~\eqref{eq-problem} blow up in $L^\infty$ in finite time. The analysis in other $L^q$ spaces, $q\in[1,\infty)$, was performed in \cite{Weissler-1979,Weissler-1980,Weissler-1981}; see also~\cite{Giga-1986}. These results have been generalized to deal with some more general local situations, where the Laplacian is replaced by a different local diffusion operator and/or the reaction term $u^p$ by a nonlinearity $f(u)$. We refer the interested reader to the surveys~\cite{DL,GL,L} and the monographs~\cite{Quittner-Souplet-book,SGKM-book}.

In the nonlocal framework, but yet with a local time derivative, the Fujita exponent in $L^\infty$  for problem~\eqref{eq-problem} with $\alpha=1$ and $\beta\in(0,1)$ is $p_f=1+2\beta/N$; see~\cite{Nagasawa-Sirao-1969}, and also ~\cite{Sugitani-1975,BLW}. Other nonlocal diffusion operators have been considered for instance in~\cite{GarciaMelian-Quiros-2010}.

Notice that the Fujita exponents in the cases with $\alpha=1$ coincide with the ones for $\alpha\in(0,1)$. However, when the time derivative is nonlocal the critical exponent $p_f$ belongs to the case where there are nonnegative and nontrivial $L^q$-solutions that exist globally as well as those blowing up in finite time.

As for nonlocal time derivatives,~\cite{L-L} considers the Fujita phenomenon for an equation similar to~\eqref{eq-problem} involving the Caputo-Hadamard time derivative, instead of a Caputo derivative. On the other hand,~\cite{Z-S} analyzes~\eqref{eq-problem} with $\alpha\in(0,1)$ when $\beta=1$ and considers a somewhat different meaning of solution. Namely, the authors look at~\eqref{eq-formula} in terms of the operator giving the solution to the Cauchy problem for the classical heat equation, and strongly use the exponential decay in time of the norm of such operator. Also, due to their approach, their results are only valid  for continuous initial data that tend to zero at infinity. Our methods allow us to work with any $\beta\in(0,1]$ and much more general initial data. Of course, the  value that we obtain for the Fujita exponent (when $\beta=1$) is the same as the one in~\cite{Z-S}.  We point out that the condition in \cite{Z-S} for global existence when $p\ge p_f$  is the strong (not sharp) one~\eqref{eq:l.norm.small}.

For the fully nonlocal case, $\alpha,\beta\in(0,1)$, some (non-sharp) blow-up results  were obtained in~\cite{Kirane-Laskri-Tatar-2005} in the $L^\infty$ framework. For the linear case $p=1$ there is a  recent paper~\cite{Dipierro-Pellacci-Valdinoci-Verzini-2021} where the authors prove the existence and uniqueness of a fundamental solution (with initial datum a Dirac mass) to~\eqref{eq-problem}, and obtain interesting asymptotic information for general solutions when $N=1$, $\alpha\in[1/2,1)$ and $\beta\in(0,1/2)$.

Let us finally mention the paper~\cite{Ferreira-dePablo-2024}, that considers blow-up issues including the Fujita phenomenon, in $L^\infty(\mathbb{R}^N)$, for nonnegative solutions of a different fully fractional heat equation, namely $(\partial_t-\Delta u)^\alpha=u^p$, $p>0$, $\alpha\in(0,1)$.

\subsection{Some open problems}

\subsubsection{Small dimensions}
In this paper we consider the case of high dimensions, $N>4\beta$. Other dimensions can be treated in a similar way, since good estimates for the kernels $Z$ and $Y$ are available; see for instance~\cite{Kemppainen-Siljander-Zacher-2017,Kim-Lim-2016}. If $N<2\beta$ both kernels are bounded, and the proofs will be more similar to the ones for the case $\alpha=1$ in which the time derivative is local. In the critical cases $N=2\beta$ and $N=4\beta$ logarithmic behaviors occur, and the proofs may be a bit more involved.

\subsubsection{Instantaneous blow-up outside the good range}
\label{subsubsect:instantaneous}

In the local case, $\alpha=1$, $\beta=1$, if $q\in[1,\ell)$ there are some initial data $u_0\in L^q(\mathbb{R}^N)$ for which problem~\eqref{eq-problem} admits no local solution (for any positive existence time) in any reasonable weak sense~\cite{Brezis-Cazenave-1996,Weissler-1980,Weissler-1981}; see also~\cite{Laister-Robinson-Sierzega-VidalLopez-2016}. We expect this also to be the case in the nonlocal setting, hence the need of extra integrability conditions on the initial datum to ensure local existence. This interesting topic deserves being studied somewhere else.

\subsubsection{Lack of uniqueness outside the good range}
\label{subsubsect:nonuniqueness}

When $q\in[1,\ell)$ and $1\le \ell\le p-1$ (which means that $p>(N+2)/(N-2)$), problem~\eqref{eq-problem} does not have uniqueness of $L^q$-solutions in the local case $\alpha=1$, $\beta=1$ \cite{Haraux-Weissler-1982}; see also~\cite{Baras-1983,Brezis-Cazenave-1996,Matos-Terraneo-2003,Ni-Sacks-1985,Terraneo-2002}. It would be interesting to study the possibility of such phenomenon in our nonlocal setting.

\subsection{Notations. } For our analysis, it will be convenient to decompose solutions to~\eqref{eq-formula} as 	
\begin{equation}
\label{eq:definition.v.w.signed}
u(x,t)=\underbrace{\int_{\R^N}u_0(y)Z(x-y,t)\,{\rm d}y}_{\mathcal{L}_{u_0}(x,t)}+\underbrace{\int_0^t\int_{\R^N} |u(y,s)|^{p-1}u(y,s)Y(x-y,t-s)\,{\rm d}y{\rm d}s}_{\mathcal{N}_u(x,t)}.
\end{equation}
The \lq\lq linear'' part of the solution, $\mathcal{L}_{u_0}$, solves the linear problem
$$
\partial_t^\a u+(-\Delta)^\beta u=0\quad\text{in }\R^N\times(0,\infty),\qquad
 	u(\cdot,0)=u_0\quad\mbox{in }\R^N.
$$
The \lq\lq nonlinear part'', $\mathcal{N}_u(x,t)$, accounts for the effect of the reaction nonlinearity, which is responsible of the possibility of blow up when $p\in(1,\infty)$.
 	  	
We denote by $p_{\rm c}:=\frac{N}{N-2\beta}$ a critical exponent that will appear several times along the paper.

\section{Estimates for the kernels}
\label{sect:estimates}
\setcounter{equation}{0}

We recall here some estimates on the kernels $Z$ and $Y$ in~\eqref{eq-formula} that will be used throughout the paper. Here and in what follows we assume always, without further mention, that $N>4\beta$.

Both kernels $Z$ and $Y$ have a selfsimilar structure,
\begin{equation}\label{eq:ss.structure}
  Z(x,t)=t^{-\frac{N\alpha}{2\beta}}F(x t^{-\frac{\alpha}{2\beta}}),\qquad  Y(x,t)=t^{-(1-\alpha+\frac{N\alpha}{2\beta})}G(x t^{-\frac{\alpha}{2\beta}}),
\end{equation}
for certain radially decreasing positive profiles $F$ and $G$ that are singular at the origin for the dimensions that we are considering, and smooth away from it.
 		
For all $t>0$, $\|Z(\cdot,t)\|_{L^1(\R^N)}=1$ and
 		\begin{align}
            \label{eq:pointwise.estimate.Z}
            0<Z(x,t)&\le K \begin{cases}t^{-\a}|x|^{2\beta-N}\quad\hskip0.24cm&\mbox{if }|x|<t^{\frac\alpha{2\beta}},\\
 			t^\alpha |x|^{-(2\beta+N)}&\mbox{if }|x|>t^{\frac\alpha{2\beta}},\end{cases}
 				\\
            \notag
 			|Z_t(x,t)|&\le K\begin{cases}t^{-(1+\alpha)}|x|^{2\beta-N}&\mbox{if }|x|<t^{\frac\alpha{2\beta}},\\
 				t^{\alpha-1}|x|^{-(2\beta+N)}&\mbox{if }|x|>t^{\frac\alpha{2\beta}}.\end{cases}
\end{align}
The estimate for $Z$ close to the origin is sharp. Hence, $Z(\cdot,t)\in L^r(\mathbb{R}^N)$ if and only if $r\in[1,p_{\rm c})$, and in such case we have
\begin{align}
\label{eq-estimate Z}
\|Z(\cdot,t)\|_{L^r(\R^N)}\le Kt^{-\frac{N\alpha}{2\beta}(1-\frac1r)},\qquad&t>0,\\
\label{eq-estimate Zt}
\|Z(\cdot,t)-Z(\cdot,t_0)\|_{L^r(\R^N)}\le K(t-t_0)t_0^{-1-\frac{N\alpha}{2\beta}(1-\frac1r)},\qquad&t\ge t_0>0.
\end{align}	

\begin{rem}
\label{rk:initial.data.Lq} Let $u_0\in L^q(\mathbb{R}^N)$. Since $Z(\cdot,t)$ is an approximation of the identity as $t\to0^+$, then
$$
\|\mathcal{L}_{u_0}(\cdot,t)-u_0\|_{L^q(\mathbb{R}^N)}\to 0^+\quad\text{as }t\to0^+.
$$
Therefore, if $u$ satisfies~\eqref{eq-formula}, to prove~\eqref{eq:initial.datum.Lq} it is enough to check that
\begin{equation}
\label{eq:initial.data.Nu}
\operatornamewithlimits{ess\,lim}\limits_{t\to0^+}\|\mathcal{N}_u(\cdot,t)\|_{L^q(\mathbb{R}^N)}=0.
\end{equation}
\end{rem}

As for the kernel $Y$,
	\begin{align*}
	0<Y(x,t)&\le K \begin{cases}t^{-(1+\a)}|x|^{4\beta-N}&\mbox{if }|x|<t^{\frac\alpha{2\beta}},\\
		t^{2\alpha-1} |x|^{-(2\beta+N)}&\mbox{if }|x|>t^{\frac\alpha{2\beta}},\end{cases}
	\\
	|Y_t(x,t)|&\le K\begin{cases}t^{-(2+\alpha)}|x|^{4\beta-N}&\mbox{if }|x|<t^{\frac\alpha{2\beta}},\\
		t^{2\alpha-2}|x|^{-(2\beta+N)}&\mbox{if }|x|>t^{\frac{\alpha}{2\beta}}.\end{cases}
\end{align*} 	
The estimate for $Y$ close to the origin is sharp. Thus, $Y(\cdot,t)\in L^r(\mathbb{R}^N)$ if and only if $r\in[1,N/(N-4\beta))$, and in such case we have
\begin{align}
\label{eq-estimate Y}
\|Y(\cdot,t-s)\|_{L^r(\R^N)}\le K(t-s)^{-\sigma(r)},\qquad&0\le s<t,\\
\label{eq-estimate Yt}
\|Y(\cdot,t-s)-Y(\cdot,t_0-s)\|_{L^r(\R^N)}\le K(t-t_0)(t_0-s)^{-\sigma(r)-1},\qquad&0\le s<t_0\le t,
\end{align}	
where
$$
\sigma(r):=1-\alpha+\frac{N\alpha}{2\beta}(1-\frac1r).
$$
Notice that $\sigma(r)< 1$ if and only if $r<p_{\rm c}$. Hence, $Y(\cdot,t)\in L^1_{\rm loc}([0,\infty);L^r (\mathbb{R}^N))$ if and only if $r\in[1,p_{\rm c})$. This restriction on $r$ will appear several times in our analysis.

For these estimates, and  many more, on $Z$ and $Y$, see for instance~\cite{Eidelman-Kochubei-2004,Kemppainen-Siljander-Zacher-2017,Kim-Lim-2016,Kochubei-1990}.

\section{Basic theory in the good range}\label{sect-local existence}
\setcounter{equation}{0}

In this section we deal with values of $q$ satisfying~\eqref{eq:good.range}. In this good range we will be able to prove the existence of a unique maximal $L^q$-solution for any initial datum in~$L^q(\mathbb{R}^N)$, which is moreover a continuous curve in $L^q(\mathbb{R}^N)$ up to the maximal existence time. When the latter is finite, $L^q$-solutions become unbounded as they approach it.

\subsection{Local existence}
We start by proving, via a fixed point argument, that  problem~\eqref{eq-problem} has an $L^q$-solution for any $u_0\in L^q(\R^N)$ if $T$ is small.

\begin{teo}\label{teo-local existence} Let $p\in[1,\infty)$ and $q$ in the good range~\eqref{eq:good.range}. Given $u_0\in L^q(\R^N)$, there exists $t_0>0$ such that problem~\eqref{eq-formula} has a solution $u\in L^\infty(0,t_0;L^q(\mathbb{R}^N))$ satisfying~\eqref{eq:initial.datum.Lq}. If $u_0\ge0$, $u$ is nonnegative. When restricted to $(0,T)$, $T\in(0,t_0]$, $u$ is an $L^q$-solution to~\eqref{eq-problem} with existence time $T$.
\end{teo}

\begin{proof}
Let $C_0=2\|u_0\|_{L^{ q }(\R^N)}$ and $t_0>0$ to be chosen later. For $v\in L^\infty(0,t_0; L^{ q }(\R^N))$,  let
\begin{equation*}\label{eq-fixed point}
		\T v(x,t)=\int_{\R^N} u_0(y)Z(x-y,t)\,{\rm d}y+\int_0^t\int_{\R^N}|v(y,s)|^{p-1}v(y,s)Y(x-y,t-s)\,{\rm d}y{\rm d}s.
\end{equation*}
Our aim is to prove that $\T$ has a fixed point in
\begin{equation*}\label{eq-fixed point set}
 	\K=\{v\in L^\infty\big(0,t_0; L^{ q }(\R^N)\big),\,\vvvert v\vvvert :=\sup_{t\in(0,t_0)}\|v(\cdot,t)\|_{L^{ q }(\R^N)}\le C_0 \}
\end{equation*}
if $t_0$ is small enough. Observe that $\K$ is a closed metric space with the metric $d(v,\hat v):=\vvvert v-\hat v\vvvert$.

We first prove that $\T(\K)\subset\K$ if $t_0$ is small enough. Let $r$ be such that
 \begin{equation}
 \label{eq:choice.r}
 1-\frac1r=\frac pq-\frac1q.
 \end{equation}
Since $p\ge 1$ and $q>\ell$, then $0\le 1-\frac1r<\frac{2\beta}N$, or equivalently, $r\in[1,p_{\rm c})$, whence estimate \eqref{eq-estimate Y} holds and $\sigma(r)<1$. Thus, for $0\le t\le t_0$,
\begin{equation}
\label{eq:in.the.space}
\begin{aligned}
 	\|\T v(\cdot,t)\|_{L^{ q }(\R^N)}&\le \|u_0\|_{L^{ q }(\R^N)}+\int_0^t\||v(\cdot,s)|^p\|_{L^{\frac  q  p}(\R^N)}\|Y(\cdot,t-s)\|_{L^r(\R^N)}\,{\rm d}s\\
 	&\le \frac{C_0}2
 	+C_0^pK\int_0^t(t-s)^{-\sigma(r)}\,{\rm d}s\le C_0\Big(\frac12+\frac{C_0^{p-1}K}{1-\sigma(r)}t_0^{1-\sigma(r)}\Big)\le C_0
 \end{aligned}
 \end{equation}
if $t_0$ is small, where we have used also that $\|Z(\cdot,t)\|_{L^1(\R^N)}=1$ for all $t>0$.

Let us check now that $\T$ is a strict contraction in $\K$ if $t_0$ is small. Let $v_1,v_2\in \K$. By \eqref{eq-estimate Y},
\[\begin{aligned}
	\|\T v_1(\cdot,t)&-\T v_2(\cdot,t)\|_{L^{ q }(\R^N)}\\
    &\le K\int_0^t\||v_1(\cdot,s)|^{p-1}v_1(\cdot,s)-|v_2(\cdot,s)|^{p-1}v_2(\cdot,s)\|_{L^{\frac q  p}(\R^N)}(t-s)^{-\sigma(r)}\,{\rm d}s\\
    &\le K \sup_{s\in(0,t_0)}\||v_1(\cdot,s)|^{p-1}v_1(\cdot,s)-|v_2(\cdot,s)|^{p-1}v_2(\cdot,s)\|_{L^{\frac q  p}(\R^N)}\frac{t^{1-\sigma(r)}}{1-\sigma(r)}
\end{aligned}\]
if $0<t\le t_0$. On the other hand, if $p>1$,
$$
\begin{aligned}\||v_1(\cdot,s)|^{p-1}v_1(\cdot,s)&-|v_2(\cdot,s)|^{p-1}v_2(\cdot,s)\|_{L^{\frac q  p}(\R^N)}\\
		&\le p\|(|v_1(\cdot,s)|+|v_2(\cdot,s)|)^{p-1}|v_1(\cdot,s)-v_2(\cdot,s)|\|_{L^{\frac q  p}(\R^N)}\\
	&\le p\||v_1(\cdot,s)|+|v_2(\cdot,s)|\|_{L^{ q }(\R^N)}^{p-1}\|v_1(\cdot,s)-v_2(\cdot,s)\|_{L^{ q }(\R^N)},
\end{aligned}
$$
so that
\begin{equation}\label{eq-estimate difference powers}\begin{aligned}
\sup_{s\in(0,t_0)}\||v_1(\cdot,s)|^{p-1}v_1(\cdot,s)-|v_2(\cdot,s)|^{p-1}v_2(\cdot,s)\|_{L^{\frac q  p}(\R^N)}\le p(2C_0)^{p-1}\|v_1(\cdot,s)-v_2(\cdot,s)\|_{L^{ q }(\R^N)},
\end{aligned} \end{equation}
an inequality which is trivially true for $p=1$. Hence, for every $p\ge 1$, $0<t\le t_0$,
\begin{equation}
\label{eq:contraction}
	\|\T v_1(\cdot,t)-\T v_2(\cdot,t)\|_{L^{ q }(\R^N)}\le pK (2C_0)^{p-1} \frac{t_0^{1-\sigma(r)}}{1-\sigma(r)}\vvvert v_1-v_2\vvvert \le \frac12 \vvvert v_1-v_2\vvvert
\end{equation}
if $t_0$ is small enough.  Hence, $\T$ has a unique fixed point $u$ in $\K$, which solves~\eqref{eq-formula}. Moreover,
\[\begin{aligned}
 	\|\mathcal{N}_u(\cdot,t)\|_{L^{ q }(\R^N)}&\le \int_0^t\||u(\cdot,s)|^p\|_{L^{\frac  q  p}(\R^N)}\|Y(\cdot,t-s)\|_{L^r(\R^N)}\,{\rm d}s\\
 	&\le C_0^pK\int_0^t(t-s)^{-\sigma(r)}\,{\rm d}s\le \frac{C_0^pK}{1-\sigma(r)}t^{1-\sigma(r)},
 \end{aligned}\]
which implies~\eqref{eq:initial.data.Nu}. Hence,~\eqref{eq:initial.datum.Lq} holds; see Remark~\ref{rk:initial.data.Lq}.

Now, assume $u_0\ge0$ and let
\[\K_0=\{v\in L^\infty\big(0,t_0; L^{ q }(\R^N)\big),\vvvert v\vvvert\le C_0 \mbox{ and }v\ge0\}
\subset \K.\]
When restricted to $\K_0$, $\T$ is still a contraction if $t_0$ is small as above, and moreover $\T(\K_0)\subset \K_0$. Since $\T$ has a unique fixed point in $\K$, and a unique fixed point in $\K_0\subset \K$, they are in fact the same one, belonging to $\K_0$.
\end{proof}	

\begin{rem}
If $p=1$, we are in the good range for all $q\in[1,\infty]$, and we can choose the time step $t_0$ in Theorem~\ref{teo-local existence} independently of the initial datum $u_0$; see~\eqref{eq:in.the.space} and~\eqref{eq:contraction}.
\end{rem}

\begin{rem}\label{rem-gr} Observe that $p\ge p_{\rm c}$ if and only if $\ell\ge p$. Therefore, when $p\ge p_{\rm c}$, the condition $q>\ell$ is enough to guarantee that $q$ is in the good range~\eqref{eq:good.range}, and hence local existence,  whereas if $p\in[1,p_{\rm c})$ we just need to require that $q\ge p$.
\end{rem}
\subsection{Continuity} We next prove that, in the good range, $L^q$-solutions are continuous curves in $L^q$ in their time interval of existence.
\begin{prop}
\label{prop:continuity}
Let $p\in[1,\infty)$,  $q$ in the good range~\eqref{eq:good.range}, and $u_0\in L^q(\mathbb{R}^N)$. If $u$ is an $L^q$-solution to~\eqref{eq-formula} in $(0,T)$, then $u\in C([0,T);L^q(\R^N))$.
\end{prop}
\begin{proof}
Let us first check that $u(\cdot,t)$ given by~\eqref{eq-formula} is well defined and belongs to $L^q(\mathbb{R}^N)$ \emph{for all} $t\in (0,T)$. Indeed, since  $\|Z(\cdot,t)\|_{L^1(\R^N)}=1$ for all $t>0$, taking $r$ as in~\eqref{eq:choice.r} and using the bound~\eqref{eq-estimate Y},
\begin{align*}
\|u(\cdot,t)\|_{L^{ q }(\R^N)}&\le \|u_0\|_{L^{ q }(\R^N)}+\int_0^t\||u(\cdot,s)|^p\|_{L^{\frac  q  p}(\R^N)}\|Y(\cdot,t-s)\|_{L^r(\R^N)}\,{\rm d}s\\
&\le  K\sup_{s\in(0,t)}\|u(\cdot,s)\|^p_{L^q(\R^N)}\int_{0}^{t}(t-s)^{-\sigma(r)}\,{\rm d}y{\rm d}s=K\sup\limits_{s\in(0,t)}\|u(\cdot,s)\|^p_{L^q(\R^N)}\frac{t^{1-\sigma(r)}}{1-\sigma(r)}.
\end{align*}

Continuity at $t=0$ comes for free by the very definition of $L^q$-solution.  To consider the continuity outside the origin we make the decomposition~\eqref{eq:definition.v.w.signed}, and prove separately the continuity of $\mathcal{L}_{u_0}$ and~$\mathcal{N}_u$.

Let $0<t_1<t_2<T$. We have, using~\eqref{eq-estimate Zt} with $r=1$,
$$
\|\mathcal{L}_{u_0}(\cdot,t_2)-\mathcal{L}_{u_0}(\cdot,t_1)\|_{L^q(\mathbb{R}^N)}\le \|u_0\|_{L^q(\mathbb{R}^N)}\|Z(\cdot,t_2)-Z(\cdot,t_1)\|_{L^1(\mathbb{R}^N)}
\le  \|u_0\|_{L^q(\mathbb{R}^N)}K(t_2-t_1)t_1^{-1},
$$
from where the desired continuity for $\mathcal{L}_{u_0}$ follows immediately. As for $\mathcal{N}_u$,  we make the decomposition
$$
\begin{aligned}
\mathcal{N}_u(\cdot,t_2)-\mathcal{N}_u(\cdot,t_1)&={\rm I}(\cdot,t_2,t_1)+{\rm II}(\cdot,t_2,t_1)+{\rm III}(\cdot,t_2,t_1),\quad\text{where}\\
{\rm I}(x,t_2,t_1)&=\int_{0}^{t_1-\varepsilon(t_2-t_1)}\int_{\R^N} |u(y,s)|^{p-1}u(y,s)\big(Y(x-y,t_2-s)-Y(x-y,t_1-s)\big)\,{\rm d}y{\rm d}s,\\
{\rm II}(x,t_2,t_1)&=\int_{t_1-\varepsilon(t_2-t_1)}^{t_1}\int_{\R^N} |u(y,s)|^{p-1}u(y,s)\big(Y(x-y,t_2-s)-Y(x-y,t_1-s)\big)\,{\rm d}y{\rm d}s,\\
{\rm III}(x,t_2,t_1)&=\int_{t_1}^{t_2}\int_{\R^N} |u(y,s)|^{p-1}u(y,s)Y(x-y,t_2-s)\,{\rm d}y{\rm d}s,
\end{aligned}
$$
with $\varepsilon>0$ small enough so that $t_1-\varepsilon(t_2-t_1)>0$. Notice that when $t_2\to t_1$  and $t_1>0$ this means no restriction for $\varepsilon$.

To estimate ${\rm I}$ we take $r$ as in~\eqref{eq:choice.r} and use the bound~\eqref{eq-estimate Yt},
$$
\begin{aligned}
\|{\rm I}(\cdot,t_2,t_1)\|_{L^q(\mathbb{R}^N)}&\le\int_{0}^{t_1-\varepsilon(t_2-t_1)}\||u(\cdot,s)|^p\|_{L^{\frac  {q}  p}(\R^N)}\|Y(\cdot,t_2-s)-Y(\cdot,t_1-s)\|_{L^r(\mathbb{R}^N)}\,{\rm d}y{\rm d}s\\
&\le  K\sup_{t\in(0,t_2)}\|u(\cdot,t)\|^p_{L^q(\R^N)}(t_2-t_1)\int_{0}^{t_1-\varepsilon(t_2-t_1)}(t_1-s)^{-\sigma(r)-1}\,{\rm d}y{\rm d}s\\
&=K\sup\limits_{t\in(0,t_2)}\|u(\cdot,t)\|^p_{L^q(\R^N)}\frac{\varepsilon^{-\sigma(r)}(t_2-t_1)^{1-\sigma(r)}-t_1^{-\sigma(r)}(t_2-t_1)}{\sigma(r)}.
\end{aligned}
$$

As for ${\rm II}$, using the bound~\eqref{eq-estimate Y},
$$
\begin{aligned}
\|{\rm II}(\cdot,t_2,t_1)\|_{L^q(\mathbb{R}^N)}&\le\int_{t_1-\varepsilon(t_2-t_1)}^{t_1}\!\||u(\cdot,s)|^p\|_{L^{\frac q p}(\R^N)}\!\big(\|Y(\cdot,t_2-s)\|_{L^r(\mathbb{R}^N)}+\|Y(\cdot,t_1-s)\|_{L^r(\mathbb{R}^N)}\big)\,{\rm d}y{\rm d}s\\
&\le K\sup_{t\in(0,t_2)}\|u(\cdot,t)\|^p_{L^q(\R^N)}\int_{t_1-\varepsilon(t_2-t_1)}^{t_1}\big((t_2-s)^{-\sigma(r)}+(t_1-s)^{-\sigma(r)}\big)\,{\rm d}y{\rm d}s\\
&=K\sup_{t\in(0,t_2)}\|u(\cdot,t)\|^p_{L^q(\R^N)}(t_2-t_1)^{1-\sigma(r)}\frac{\big((1+\varepsilon)^{1-\sigma(r)}-1+\varepsilon^{1-\sigma(r)}\big)}{1-\sigma(r)}.
\end{aligned}
$$

In order to estimate ${\rm III}$ we use again the bound~\eqref{eq-estimate Y},

\begin{align*}
\|{\rm III}(\cdot,t_2,t_1)\|_{L^q(\mathbb{R}^N)}&\le\int_{t_1}^{t_2}\int_{\R^N}\||u(\cdot,s)|^p\|_{L^{\frac  q p}(\R^N)}\|Y(\cdot,t_2-s)_{L^r(\mathbb{R}^N)}\,{\rm d}y{\rm d}s\\
&\le K\sup_{t\in(0,t_2)}\|u(\cdot,t)\|^p_{L^q(\R^N)}\int_{t_1}^{t_2}(t_2-s)^{-\sigma(r)}\,{\rm d}y{\rm d}s\\
&=K\sup_{t\in(0,t_2)}\|u(\cdot,t)\|^p_{L^q(\R^N)}\frac{(t_2-t_1)^{1-\sigma(r)}}{1-\sigma(r)}.\qedhere
\end{align*}
\end{proof}
\begin{rem}
If $\sup\limits_{t\in(0,T)}\|u(\cdot, t)\|_{L^q(\R^N)}<\infty$, we may repeat the arguments in the proof of Proposition~\ref{prop:continuity} to check first that
$$
u(\cdot,T):=\int_{\R^N} u_0(y)Z(x-y,T)\,{\rm d}y+\int_0^T\int_{\R^N}|u(y,s)|^{p-1}u(y,s)Y(x-y,T-s)\,{\rm d}y{\rm d}s
$$
is well defined and belongs to $L^q(\mathbb{R}^N)$, and then that
$\lim\limits_{t\nearrow T} u(\cdot, t)=u(\cdot,T)$ in $L^q(\mathbb{R}^N)$.
\end{rem}

\subsection{Extension}\label{sect-extension}

We now show that if the $L^q$-norm of an $L^q$-solution is bounded up to its existence time $T$, then we can construct an $L^q$-solution with a larger existence time.

\begin{teo}\label{teo-extension}
Let $p\in[1,\infty)$ and $q$ in the good range~\eqref{eq:good.range}. Let $u$ be an $L^q$-solution to~\eqref{eq-problem} with initial datum $u_0\in L^q(\R^N)$ and existence time $T$. If $\sup_{t\in(0,T)}\|u(\cdot, t)\|_{L^q(\R^N)}<\infty$, then there exists some $\tau>0$ such that~\eqref{eq-problem} has an $L^q$-solution $\bar u$ with initial datum $u_0\in L^q(\R^N)$ and existence time $\bar T=T+\tau$. Moreover, $\sup_{t\in(0,\bar T)}\|\bar u(\cdot, t)\|_{L^q(\R^N)}<\infty$. If $u\ge0$, then $\bar u\ge0$.
\end{teo}
\begin{proof}

The idea is to construct a solution $\bar u$ that coincides with $u$ for $t\in(0,T-\tau)$ for some $\tau\in (0,T/2)$ small to be further specified later. Thus, for $t\in(T-\tau,T+\tau)$ we should have
\begin{align*}
\bar u(x,t)&=\int_{\R^N} u_0(y)Z(x-y,t)\,{\rm d}y+\int_0^{T-\tau}\int_{\R^N}|u(y,s)|^{p-1}u(y,s)Y(x-y,t-s)\,{\rm d}y{\rm d}s
\\
&\quad+\int_{T-\tau}^t\int_{\R^N}|\bar u(y,s)|^{p-1}\bar u(y,s)Y(x-y,t-s)\,{\rm d}y{\rm d}s.
\end{align*}
With this in mind, we consider the set
\[
\bar\K:=\{v\in L^\infty\big(T-\tau,T+\tau;L^q(\R^N)\big),\, \vvvert v\vvvert :=\sup_{t\in(T-\tau,T+\tau)}\|v(\cdot,t)\|_{L^{ q }(\R^N)}\le C_1\},
\]
where $C_1:=2\sup\limits_{t\in(0,T)}\|u(\cdot, t)\|_{L^q(\R^N)}$, and define in it an operator $\bar\T$ by
\begin{equation}
\label{eq:definition.barT}
\begin{aligned}
\bar\T v(x,t)&=\int_{\R^N} u_0(y)Z(x-y,t)\,{\rm d}y+\int_0^{T-\tau}\int_{\R^N}|u(y,s)|^{p-1}u(y,s)Y(x-y,t-s)\,{\rm d}y{\rm d}s
\\
&\quad+\int_{T-\tau}^t\int_{\R^N}|\bar u(y,s)|^{p-1}\bar u(y,s)Y(x-y,t-s)\,{\rm d}y{\rm d}s.
\end{aligned}
\end{equation}
We will prove that $\bar\T$ has a (unique) fixed point $\tilde u$ in $\bar\K$ if $\tau>0$ is small enough. Then,
\[\bar u(x,t)=\begin{cases}
	u(x,t),&t\in(0,T-\tau),\\
	\tilde u(x,t),&t\in(T-\tau,T+\tau),
	\end{cases}\]
will be an $L^q$-solution to \eqref{eq-problem} with existence time $\bar T=T+\tau$. 	

To prove that $\bar\T(\bar\K)\subset \bar\K$ if  $\tau$ is small enough, we rewrite $\bar\T v$ as $\bar\T v ={\rm I}+{\rm II}+{\rm III}+{\rm IV}$+{\rm V}, where
\begin{align*}
{\rm I}(x)&=\int_{\R^N} u_0(y)Z(x-y,T)\,{\rm d}y+\int_0^T\int_{\R^N}|u(y,s)|^{p-1}u(y,s)Y(x-y,T-s)\,{\rm d}y{\rm d}s,\\
{\rm II}(x,t)&=\int_{\R^N} u_0(y)\big(Z(x-y,t)-Z(x-y,T)\big)\,{\rm d}y,\\
{\rm III}(x,t)&=\int_0^{T-\tau}\int_{\R^N} |u(y,s)|^{p-1}u(y,s)\big(Y(x-y,t-s)-Y(x-y,T-s)\big)\,{\rm d}y{\rm d}s,\\
{\rm IV}(x,t)&=\int_{T-\tau}^{T}\int_{\R^N}|u(y,s)|^{p-1}u(y,s)Y(x-y,T-s)\,{\rm d}y{\rm d}s,\\
{\rm V}(x,t)&=\int_{T-\tau}^t\int_{\R^N} |v(y,s)|^{p-1}v(y,s)Y(x-y,t-s)\,{\rm d}y{\rm d}s.
\end{align*}
On the one hand $\|{\rm I}\|_{L^{ q }(\R^N)}=\|u(\cdot, T)\|_{L^q(\R^N)}\le C_1/2$.
On the other hand, by Young's inequality  and~\eqref{eq-estimate Zt} with $r=1$, if $\tau\in(0,T/(8K))$ we have
\begin{align*}
\|{\rm II}(\cdot,t)\|_{L^{ q }(\R^N)}&\le \|u_0\|_{L^{ q }(\R^N)}\|Z(\cdot,t)-Z(\cdot, T)\|_{L^1(\R^N)}
\le \frac{KC_1}2
\begin{cases}
(t-T)T^{-1},&t\in(T,T+\tau),
\\
(T-t)t^{-1},&t\in(T-\tau,T),
\end{cases}
\\
&\le KC_1 T^{-1} \tau\le \frac{C_1}8 \quad\text{for all }t\in (T-\tau,T+\tau).
\end{align*}

In order to estimate ${\rm III}$, ${\rm IV}$ and ${\rm V}$, we take $r$ as in~\eqref{eq:choice.r}, so that, since we are in the good range, $r\in[1,p_{\rm c})$ and $\sigma(r)\in(0,1)$. Young's inequality plus~\eqref{eq-estimate Yt}, in the estimate for ${\rm III}$, or~\eqref{eq-estimate Y}, in the estimates for ${\rm IV}$ and ${\rm V}$, yield
\begin{align*}
\|{\rm III}(\cdot,t)\|_{L^{ q }(\R^N)}&\le C_1^pK\tau\int_0^{T-\tau}(T-s)^{-\sigma(r)-1}\,{\rm d}s\le C_1^pK\frac{\tau^{1-\sigma(r)}}{\sigma(r)}\le \frac{C_1}8,
\\	
\|{\rm IV}(\cdot,t)\|_{L^{ q }(\R^N)}&\le C_1^pK\int_{T-\tau}^{T}(T-s)^{-\sigma(r)}\,{\rm d}s= C_1^pK\frac{\tau^{1-\sigma(r)}}{1-\sigma(r)}\le \frac{C_1}8,\\
\|{\rm V}(\cdot,t)\|_{L^{ q }(\R^N)}&\le C_1^pK\int_{T-\tau}^t (t-s)^{-\sigma(r)}\,{\rm d}s= C_1^pK\frac{(t-T+ \tau)^{1-\sigma(r)}}{1-\sigma(r)}\le C_1^pK\frac{(2\tau)^{1-\sigma(r)}}{1-\sigma(r)}\le\frac{C_1}8
\end{align*}
if $\tau$ is small enough, namely, if
$$
\tau\le\min\Big\{\Big(\frac{C_1^{1-p}\sigma(r)}{8K}\Big)^{\frac1{1-\sigma(r)}},\frac12\Big(\frac{C_1^{1-p}(1-\sigma(r))}{8K}\Big)^{\frac1{1-\sigma(r)}}
\Big\}.
$$

Now that we know that $\bar \T(\bar \K)\subset\bar \K$, we will prove that $\bar\T$ is a contraction in $\bar\K$ for the distance associated to the norm $\vvvert v\vvvert$ if $\tau$ is small. Let $v_1,\,v_2\in\bar\K$. For $t\in(T-\tau,T+\tau)$,
\[\begin{aligned}
	\|\bar\T v_1(\cdot,t)&-\bar\T v_2(\cdot t)\|_{L^{ q }(\R^N)}\\
	&\le \int_{T- \tau}^t\||v_1(\cdot,s)|^{p-1}v_1(\cdot,s)-|v_2(\cdot,s)|^{p-1}v_2(\cdot,s)\|_{L^{\frac q  p}(\R^N)}\|Y(\cdot,t-s)\|_{L^r(\R^N)}\,{\rm d}s.
\end{aligned}\]
Hence, by \eqref{eq-estimate difference powers} with $C_0$ replaced by $C_1$,
if $\tau\le\frac12\Big(\frac{(2C_1)^{1-p}(1-\sigma(r))}{2pK}\Big)^{\frac1{1-\sigma(r)}}$, then	
\[\begin{aligned}
	\|\bar\T v_1(\cdot,t)-\bar\T v_2(\cdot, t)\|_{L^{ q }(\R^N)}&\le p(2C_1)^{p-1}\vvvert v_1-v_2\vvvert K\int_{T-\tau}^t(t-s)^{-\sigma(r)}\,ds\\
	&= p(2C_1)^{p-1}K\frac{(t-T+\tau)^{1-\sigma(r)}}{1-\sigma(r)}\vvvert v_1-v_2\vvvert \\
			&\le  p(2C_1)^{p-1}K\frac{(2\tau)^{1-\sigma(r)}}{1-\sigma(r)}\vvvert v_1-v_2\vvvert <\frac12 \vvvert v_1-v_2\vvvert.
		\end{aligned}\]

Let now $u_0\ge0$ and
\[\bar\K_0:=\{v\in L^\infty(T-\tau,T+\tau;L^{ q }(\R^N))\,,\vvvert v\vvvert \le C_1\mbox{ and }v\ge0\}.\]
Let us see that if $u\ge0$ and $\tau$ is chosen as above, then $\bar\T:\bar\K_0\to\bar\K_0$ and it is a contraction. To this aim we just have to see that $v\ge0$ implies that $\T v\ge0$. This follows immediately from the definition of $\bar\T$; see~\eqref{eq:definition.barT}.\end{proof}

\begin{rem}
If $p=1$ the increment $\tau$ of the existence time in Theorem~\ref{teo-extension} can be chosen independently of $\sup\limits_{t\in(0,T)}\|u(\cdot, t)\|_{L^{ q }(\R^N)}$.
\end{rem}

\begin{rem}
By construction,
$\bar u(\cdot,t)=	u(\cdot,t)$ for $t\in(0,T-\tau)$. In the next paragraph we will prove that in fact the equality holds up to time $T$.
\end{rem}

\subsection{Uniqueness}

In the good range any two $L^q$-solutions with the same initial datum coincide in their common existence time. In particular, there is at most one $L^q$-solution with a given existence time.
\begin{teo}\label{teo-local uniqueness}
Let $p\in[1,\infty)$ and $q$ in the good range~\eqref{eq:good.range}. Let $u_1$ and $u_2$ be two $L^q$-solutions to~\eqref{eq-problem} with initial datum $u_0\in L^q(\R^N)$ and corresponding existence times $T_1$ and $T_2$. Then $u_1(\cdot,t)=u_2(\cdot,t)$ for all $t\in(0,\min\{T_1,T_2\})$.
 \end{teo}
\begin{proof}
Let $\bar t\in(0,\min\{T_1,T_2\})$. We define $\K$ and $\T$ as in the proof of Theorem~\ref{teo-local existence}, but now with
$$
C_0=\max\{2\|u_0\|_{L^q(\R^N)},\sup\limits_{t\in(0,\bar t)}\|u_1(\cdot,t)\|_{L^q(\R^N)},\sup\limits_{t\in(0,\bar t)}\|u_2(\cdot,t)\|_{L^q(\R^N)}\}
$$
Repeating that proof,  we see that if $t_0>0$ is small enough then $\T:\K\to\K$ is a contraction. Therefore, $\T$ has a unique fixed point in $\K$. As $u_1$ and $u_2$ belong to $\K$ if $t_0\le \bar t$ and are fixed points of $\T$, then $u_1(\cdot,t)=u_2(\cdot,t)$ if $0<t\le t_0$.

Let $\widehat T:=\sup\{ \hat t\in(0,\min\{T_1,T_2\}): u_1(\cdot,t)=u_2(\cdot,t)\text{ for }0<t\le \hat t\}$. Assume for contradiction that $\widehat T<\min\{T_1,T_2\}$. Let us call $u(x,t)=u_1(x,t)=u_2(x,t)$ for $0<t<\widehat T$ and
$$
C_1=\max\{2\sup\limits_{t\in(0,\widehat T)}\|u(\cdot, t)\|_{L^q(\R^N))},\sup\limits_{t\in(0,\widehat T)}\|u_1(\cdot,t)\|_{L^q(\R^N)},\sup\limits_{t\in(0,\widehat T)}\|u_2(\cdot,t)\|_{L^q(\R^N)}\}.
$$
Let $\bar\K$ and $\bar\T$ as in Theorem \ref{teo-extension} with $ T$ replaced by $\widehat T$. As in that theorem, we see that $\bar\T:\bar\K\to\bar\K$ is a contraction  if $\tau$ is small enough, so that $\bar\T$ has a unique fixed point in $\bar\K$.
Besides, if in addition~$\tau$ is such that $\widehat T+\tau\le \min\{T_1,T_2\}$, then, $u_1$ and $u_2$ belong to $\bar\K$ and they are both fixed points of~$\bar \T$. Hence, $u_1(\cdot,t)=u_2(\cdot,t)$ for $0<t\le\widehat T+\tau$. But this contradicts the definition of $\widehat T$. Hence, $\widehat T=\min\{T_1,T_2\}$.	
\end{proof}

\subsection{Maximal solutions}
	
If $q$ is in the good range~\eqref{eq:good.range}, we define the \emph{maximal existence time} of $L^q$-solutions to~\eqref{eq-problem} with initial datum $u_0\in L^q(\mathbb{R}^N)$ by
$$
T_m(u_0):=\sup\{T>0: \exists \mbox{ an $L^q$-solution to~\eqref{eq-problem} with initial datum $u_0$ and existence time $T$}\}.
$$
In this range, an $L^q$-solution with existence time $T$ belongs to $C([0,T);L^q(\mathbb{R}^N))$, so that it makes sense to consider its value (in $L^q(\mathbb{R}^N)$) at a time $t\in[0,T)$. Given $t\in[0,T_m)$ we define $u_m(\cdot,t)=u(\cdot,t)$, where $u$ is any $L^q$-solution to~\eqref{eq-problem} with initial datum $u_0$ and existence time $T>t$.  Thanks to Theorem~\ref{teo-local uniqueness}, $u_m$ is well defined, and it is obviously an $L^q$-solution to~\eqref{eq-problem} with existence time $T_m$. Hence, by Proposition~\ref{prop:continuity} $u_m\in C([0,T_m);L^q(\mathbb{R}^N)$. This function $u_m$ is known as the \emph{maximal} $L^q$-solution to~\eqref{eq-problem} with initial datum $u_0$.

As corollaries to~Theorem~\ref{teo-extension},  we have the following results for maximal $L^q$-solutions in the good range.
\begin{coro}
\label{coro-global existence}
Let $p=1$, $q\in[1,\infty]$ and $u_0\in L^q(\R^N)$. The unique maximal $L^q$-solution to problem~\eqref{eq-problem} with initial datum $u_0$ is global in time. 	
\end{coro}

\begin{coro}\label{coro-blow up}
Let $p\in(1,\infty)$, $q$ in the good range~\eqref{eq:good.range}, and $u_0\in L^q(\R^N)$. Let $u$ be the maximal $L^q$-solution to~\eqref{eq-problem} with initial datum $u_0\in L^q$ and $T$ its maximal existence time. If $T<\infty$, then
\begin{equation*}\label{eq-blow up}
\limsup_{t\nearrow T}\|u(\cdot,t)\|_{L^{ q }(\R^N)}=+\infty.
\end{equation*}
\end{coro}

\subsection{More general nonlinearities}

Arguments similar to the ones that we have used in the previous paragraphs allow to deal with some nonlinearities that are not powers. A case that only requires obvious changes in the proofs is that of globally Lipschitz nonlinearities $f$ such that $f(0)=0$.
\begin{teo}\label{teo-Lipschitz} Let $f:\R\to\R$ be a  Lipschitz continuous function such that $f(0)=0$. Let $u_0\in L^ q (\R^N)$ for some $ q \ge1$. Then, there exists a unique global $L^q$-solution to
	\begin{equation}\label{eq-problem f}
\partial_t^\alpha u+(-\Delta)^\beta u=f(u)\quad\mbox{in }\R^N\times(0,\infty),\qquad
		u(\cdot,0)=u_0\quad\mbox{in }\R^N,
\end{equation}
which belongs moreover to $C([0,\infty);L^\infty(\mathbb{R}^N))$.
\end{teo}

If $f$ is yet globally Lipschitz, but $f(0)\neq0$. we still have a global existence result if the initial datum is bounded.
\begin{teo}\label{teo-bounded} Let $f:\R\to\R$ be a Lipschitz continuous function  and let $u_0\in L^\infty(\R^N)$. Then, there exists a unique global $L^\infty$-solution to \eqref{eq-problem f}, which belongs moreover to $C([0,\infty);L^\infty(\mathbb{R}^N))$.
\end{teo}
\begin{proof}
The first step is to prove existence for small times. We proceed as in Theorem \ref{teo-local existence}, now with $C_0=\max\{2\|u_0\|_{L^\infty(\R^N)};|f(0)|\}$, and find that $\T:\K\to\K$ is a contraction if $t_0$ is small, how small not depending on  the value of $C_0$. Indeed, let $L$ be a Lipschitz constant for $f$. Then, if $v\in\K$,  and $0< t\le t_0\le\big(\frac\alpha{2(1+L)}\big)^{1/\alpha}$,
\begin{equation*}
\label{eq-1st bound}
\|\T v(\cdot,t)\|_{L^\infty(\R^N)}\le \frac{C_0}2+\big(|f(0)|+LC_0\big)\frac{t^\alpha}\alpha\le C_0\Big(\frac12+(1+L)\frac{t_0^\alpha}\alpha\Big)\le C_0.
\end{equation*}
On the other hand, if $v_1,v_2\in \K$ and $0<t\le t_0< 	\big(\frac\alpha{2L}\big)^{1/\alpha}$,
\begin{equation*}
\label{eq-2nd bound}
\begin{aligned}\|\T v_1(\cdot,t)-\T v_2(\cdot,t)\|_{L^\infty(\R^N)}&\le L\frac{t_0^\alpha}\alpha \sup_{s\in(0,t_0)}\|v_1(\cdot,s)-v_2(\cdot,s)\|_{L^\infty(\R^N)}\\&<\frac12 \sup_{s\in(0,t_0)}\|v_1(\cdot,s)-v_2(\cdot,s)\|_{L^\infty(\R^N)}.
\end{aligned}
\end{equation*}
	
In a similar way we can prove a result as the one in Theorem \ref{teo-extension}. As $\tau$ can be taken independently of $C_1=\max\big\{2\sup_{t\in(0,T)}\|u(\cdot,t)\|_{L^\infty(\R^N)};|f(0)|\big\}$, by iterating the result we find that the solution is global.
	
Uniqueness follows easily as in Theorem \ref{teo-local uniqueness}, and continuity as in Proposition~\ref{prop:continuity}.
\end{proof}

When $f$ is only locally Lipschitz, we cannot ascertain global existence, but we can yet show the existence of a unique maximal $L^\infty$-solution to \eqref{eq-problem f} for any $u_0\in L^\infty(\mathbb{R}^N)$. We omit the proof, since it is an easy adaptation of previous arguments.
\begin{teo}\label{teo-existence local lipschitz}
Let $f:\R\to\R$ be locally Lipschitz continuous and $u_0\in L^\infty(\R^N)$. Then, there exists a unique maximal $L^\infty$-solution to~\eqref{eq-problem f}. Moreover, if the maximal existence time $T$ is finite, then
\[
    \limsup_{t\nearrow T}\|u(\cdot,t)\|_{L^\infty(\R^N)}=+\infty.
\]
 \end{teo}

\section{Basic theory outside the good range}
\label{sect:basic.theory.outside.GR}
\setcounter{equation}{0}

There are two reasons why $q$ may fail to be in the good range: either $p\in[1,p_{\rm c})$ and $q\in[1,p)$, or $p\in[p_c,\infty)$ and  $q\in[1,\ell]$ (see Remark \ref{rem-gr}). In order to guarantee the existence of an $L^q$-solution in these cases we require the initial data to satisfy an extra integrability condition.

If   $p\in[1,p_{\rm c})$, $q\in[1,p)$,  and $u_0\in L^p(\mathbb{R}^N)$, we know from the previous section that there exists a maximal $L^p$-solution to~\eqref{eq-problem}. If in addition $u_0\in L^q(\mathbb{R}^N)$, then this $L^p$-solution is also an $L^q$-solution, as we show next.
\begin{teo}
\label{thm:local.p.subcritical}
Let $p\in[1,p_{\rm c})$, $q\in[1,p)$, and $u_0\in L^p(\mathbb{R}^d)\cap L^q(\mathbb{R}^d)$. If $u$ is an $L^p$-solution to~\eqref{eq-problem}, it is also an $L^q$-solution. Moreover, $u\in C([0,T);L^q(\R^N))$.
\end{teo}

\begin{proof}
We make the decomposition~\eqref{eq:definition.v.w.signed}. On the one hand,
\begin{equation*}
\label{eq:bound.mathcalL}
\|\mathcal{L}_{u_0}(\cdot,t)\|_{L^q(\mathbb{R}^N)}\le \|u_0\|_{L^q(\mathbb{R}^N)}\|Z(\cdot,t)\|_{L^1(\mathbb{R}^N)}\le \|u_0\|_{L^q(\mathbb{R}^N)}.
\end{equation*}
Hence, to check that $u$ is an $L^q$-solution it is enough to show  that $\mathcal{N}_u\in L^\infty(0,\tau;L^{q}(\R^N))$ for all $\tau\in (0,T)$, and~\eqref{eq:initial.data.Nu}; see Remark~\ref{rk:initial.data.Lq}.

Let $\tau\in(0,T)$. By hypotheses, $u\in L^\infty(0,\tau;L^p(\R^N))$.  Since $p\in[1,p_{\rm c})$, then $q\in [1,p_{\rm c})$, so that $\sigma(q)<1$. Therefore,
\begin{align*}
 	\|\mathcal{N}_u(\cdot,t)\|_{L^q(\R^N)}&\le \int_0^t\||u(\cdot,s)|^p\|_{L^1(\R^N)}\|Y(\cdot,t-s)\|_{L^q(\R^N)}\,{\rm d}s\\
 &\le K\sup_{s\in(0,\tau)}\|u(\cdot,s)\|_{L^p(\R^N)}^p \int_0^t(t-s)^{-\sigma(q)}\,{\rm d}s\le K\sup_{s\in(0,\tau)}\|u(\cdot,s)\|_{L^p(\R^N)}^p \frac{t^{1-\sigma(q)}}{1-\sigma(q)}
\end{align*}
if $t\in(0,\tau)$, which implies both $\mathcal{N}_u\in L^\infty(0,\tau;L^{q}(\R^N))$, and~\eqref{eq:initial.data.Nu}.

The proof of continuity is similar to that of Proposition~\ref{prop:continuity}. We omit the details.
\end{proof}

If $p\in[p_c,\infty)$, $q\in[1,\ell]$, and $u_0\in L^{\hat q}(\R^N)$ for some $\hat q$ in the distinguished range~\eqref{eq:distinguished.range},
we know from the results in Section~\ref{sect-local existence} that there exists a maximal $L^{\hat q}$-solution to~\eqref{eq-problem}. If in addition $u_0\in L^q(\mathbb{R}^N)$, we will show that this $L^{\hat q}$-solution is also an $L^q$-solution.
\begin{rem} If $p\in[p_c,\infty)$, $q\in[1,\ell]$, and $u_0\in L^q(\R^N)\cap L^{\tilde q}(\mathbb{R}^N)$ for some $\tilde q$ in the good range~\eqref{eq:good.range}, then $u_0\in L^{\hat q}(\R^N)$ for some $\hat q$ in the distinguished range.
\end{rem}
\begin{teo}
\label{thm:local.p.superfujita}
Let $p\in[p_c,\infty)$,  $q\in[1,\ell]$, $\hat q$ in the distinguished range~\eqref{eq:distinguished.range},
and $u_0\in L^q(\mathbb{R}^d)\cap L^{\hat q}(\mathbb{R}^d)$.  If $u$ is an $L^{\hat q}$-solution to~\eqref{eq-problem}, it is also an $L^q$-solution.  Moreover, $u\in C([0,T);L^q(\R^N))$.
\end{teo}

\begin{proof}
As in the proof of Theorem~\ref{thm:local.p.subcritical}, to see that $u$ is an $L^q$-solution it is enough to check that  $\mathcal{N}_u\in L^\infty(0,\tau;L^{q}(\R^N))$ for all $\tau\in (0,T)$, and~\eqref{eq:initial.data.Nu}.

We first prove the result for $q=\ell$. To this aim we take $r$ such that $1-\frac1r=\frac p{\hat q}-\frac1\ell$. Since $\hat q\in(\ell, \ell p]$,
\[
    0\le 1-\frac1r=\frac p{\hat q}-\frac1\ell< \frac{p-1}{\ell}=\frac{2\beta}N,
\]
or equivalently, $r\in[1,p_{\rm c})$. Hence, estimate \eqref{eq-estimate Y} holds and $\sigma(r)<1$. Then, since $u\in L^\infty(0,\tau;L^{\hat q}(\R^N))$ for all $\tau\in (0,T)$ by hypothesis,
\begin{align*}
 	\|\mathcal{N}_u(\cdot,t)\|_{L^\ell(\R^N)}&\le \int_0^t\||u(\cdot,s)|^p\|_{L^{\frac{\hat q}p}(\R^N)}\|Y(\cdot,t-s)\|_{L^r(\R^N)}\,{\rm d}s\\
 &\le K\sup_{s\in(0,\tau)}\|u(\cdot,s)\|^p_{L^{\hat q}(\R^N)} \int_0^t(t-s)^{-\sigma(r)}\,{\rm d}s\le K\sup_{s\in(0,\tau)}\|u(\cdot,s)\|^p_{L^{\hat q}(\R^N)} \frac{t^{1-\sigma(r)}}{1-\sigma(r)}
\end{align*}
if $t\in(0,\tau)$, which implies both $\mathcal{N}_u\in L^\infty(0,\tau;L^{\ell}(\R^N))$, and~\eqref{eq:initial.data.Nu}.
The proof of the continuity in $L^\ell$ of an $L^{\hat q}$-solution is similar to that of Proposition~\ref{prop:continuity}.

Let now $q\in[\frac\ell p,\ell)$. Since $u_0\in L^{\hat q}(\mathbb{R}^d)\cap L^q(\mathbb{R}^d)$, then $u_0\in L^{\hat q}(\mathbb{R}^d)\cap L^\ell(\mathbb{R}^d)$. Hence, as we have just proved, the $L^{\hat q}$-solution is an $L^\ell$-solution. To prove that it is also an $L^q$-solution we take $r$ such that and $1-\frac1r=\frac p\ell-\frac1 q $. In the range of $q$ that we are considering $r\in[1,p_{\rm c})$, and therefore
$$
\|\mathcal{N}_u(\cdot,t)\|_{L^ q (\R^N)}\le K\sup_{s\in(0,\tau)}\|u(\cdot,s)\|^p_{L^\ell(\R^N)} \frac{t^{1-\sigma(r)}}{1-\sigma(r)}
$$
if $t\in(0,\tau)$, hence the result. The proof of the continuity in $L^q$ of an $L^{\ell}$-solution for $q$ in this range follows easily.

To cover the whole range down to $q=1$ we proceed by induction. Assume that we have already proved the result for $q\in[\frac\ell{p^j},\frac\ell{p^{j-1}})$ for some $j\in\mathbb{N}$. We will prove that then it is also valid for $q\in[\frac\ell{p^{j+1}},\frac\ell{p^{j}})$. Notice that, by hypothesis, $p\ge p_c>1$. Hence, $\frac{\ell}{p^j}\to0$ as $j\to\infty$, which means that we will reach $q=1$ in a finite number of steps.

Let $q\in[\frac\ell{p^{j+1}},\frac\ell{p^{j}})$.  Then $qp\in[\frac\ell{p^j},\frac\ell{p^{j-1}})$. Since $u_0\in L^{\hat q}(\mathbb{R}^d)\cap L^q(\mathbb{R}^d)$, then $u_0\in L^{\hat q}(\mathbb{R}^d)\cap L^{qp}(\mathbb{R}^d)$,  and, by the induction hypothesis, the $L^{\hat q}$-solution is an $L^{qp}$-solution. Therefore, using~\eqref{eq-estimate Y},
\begin{align*}
 	\|\mathcal{N}_u(\cdot,t)\|_{L^q(\R^N)}&\le \int_0^t\||u(\cdot,s)|^p\|_{L^q(\R^N)}\|Y(\cdot,t-s)\|_{L^1(\R^N)}\,{\rm d}s\\
 &\le K\sup_{s\in(0,\tau)}\|u(\cdot,s)\|^p_{L^{qp}(\R^N)} \int_0^t(t-s)^{-1+\alpha}\,{\rm d}s\le K\sup_{s\in(0,\tau)}\|u(\cdot,s)\|^p_{L^{qp}(\R^N)} \frac{t^\alpha}{\alpha}
\end{align*}
if $t\in(0,\tau)$, hence the result. The proof of the continuity in $L^q$ of an $L^{qp}$-solution is similar to that of Proposition~\ref{prop:continuity}.
\end{proof}

\section{Global existence above the Fujita exponent}\label{sect-global existence}
\setcounter{equation}{0}

If $p\in[p_f,\infty)$, which implies that $\ell\in[1,\infty)$, we will construct global $L^q$-solutions, for certain initial data, for all $q\in[1,\infty]$.

\subsection{The distinguished range}

We start by considering values of $q$ in the distinguished range~\eqref{eq:distinguished.range}.
\subsubsection{The strategy}
If $q$ satisfies~\eqref{eq:distinguished.range}, it is in the good range~\eqref{eq:good.range}. Hence, given $u_0\in L^q(\mathbb{R}^N)$, there is a (unique) maximal $L^q$-solution with initial datum $u_0$, with a maximal existence time $T\in(0,+\infty]$. Moreover, if $T<+\infty$,  then  $\limsup_{t\nearrow T}\|u(\cdot,t)\|_{L^q(\R^N)}=\infty$. Hence, to prove that an $L^q$-solution is global in time, it is enough to prove that $\|u(\cdot,t)\|_{L^q(\R^N)}$ is bounded by a constant that does not depend on $t$ up to time $T$. This will guaranteed if we are able to obtain such a bound for the function
\begin{equation}
\label{eq:definition.f}
f:(0,T)\to[0,\infty),\qquad f(t):=\sup_{s\in(0,t)}\|s^b u(\cdot, s)\|_{L^{ q }(\R^N)},\qquad b:=\frac{N\alpha}{2\beta}\Big(\frac1\ell-\frac1{ q }\Big)=\frac\alpha{p-1}-\frac{N\alpha}{2\beta  q },
\end{equation}
since $b>0$ when $q>\ell$. We will show that there is some value $\eta>0$ such that if
\begin{equation}
\label{eq:eta.condition}
t^b\|\mathcal{L}_{u_0}(\cdot,t)\|_{L^{ q }(\R^N)}\le \eta\quad\text{for every }t>0,
\end{equation}
then the function $f$ will never reach the value $2\eta$. The goal will be then to obtain conditions on the initial data ensuring that~\eqref{eq:eta.condition} holds.

\subsubsection{The basic lemma}
Following the above strategy, we first prove that~\eqref{eq:eta.condition} implies the desired bound for $f$.
\begin{lema}\label{lem:key.lemma}
Let $p\in[p_f,\infty)$ and $q$ satisfying~\eqref{eq:distinguished.range}. There is a value $\eta>0$ depending only on $p,q,\alpha,\beta$ and $N$ such that if $u$ is a maximal $L^q$-solution to~\eqref{eq-problem} with initial datum $u_0\in L^q(\R^N)$ satisfying~\eqref{eq:eta.condition}, and $f$ is as in~\eqref{eq:definition.f}, then $f(t)\le 2\eta$ for all $t\in(0,T)$. As a corollary, $u$ is global in time.
\end{lema}
\begin{proof}	
We borrow ideas from the paper~\cite{Weissler-1981}, which deals with the local case $\alpha,\beta=1$. But we have to face the extra difficulty stemming from the singularity of the kernels $Z$ and $Y$ at $x=0$ that appears when $\alpha\in(0,1)$.

Let $1-\frac1r=\frac p q -\frac1 q =\frac{2\beta}N\frac\ell q$. Since $p\ge p_f>1$ and $q>\ell$, then $1-\frac1r\in(0,\frac{2\beta}N)$, so that $r\in (1,p_{\rm c})$. Hence, we may use the bound~\eqref{eq-estimate Y} and the hypothesis~\eqref{eq:eta.condition} to obtain
\[\begin{aligned}
	t^b\|u(\cdot,t)\|_{L^{ q }(\R^N)}&\le t^b\|\mathcal{L}_{u_0}(\cdot,t)\|_{L^{ q }(\R^N)}+K t^b \int_0^t\|u(\cdot,s)\|_{L^{ q }(\R^N)}^p(t-s)^{\alpha-1-\frac{N\alpha}{2\beta}(1-\frac1r)} \,{\rm d}s\\
	&\le \eta+K t^{b-a} \int_0^t\|s^bu(\cdot,s)\|_{L^{ q }(\R^N)}^p s^{-bp}\Big(1-\frac st\Big)^{-a}\,{\rm d}s,
	 \end{aligned}\]
with
$$
a:= 1-\alpha+\frac{N\alpha }{2\beta  q }(p-1)=1-\alpha+\alpha\frac\ell  q.
$$
Notice that $a\in(0,1)$, since $\alpha\in(0,1)$ and $q>\ell$.
After a change of variables we  arrive at
\begin{equation}
\label{eq:inequality.f}
f(t)\le \eta+Kf(t)^p \int_0^1 \tau^{-bp}(1-\tau)^{-a}\,{\rm d}\tau,
\end{equation}
where we have used that $b-a+1-bp=0$, an identity that can be easily checked. Remember that $b>0$, because $q>\ell$. Moreover, since $q \le\ell p$,
$$
bp=\frac{N\alpha}{2\beta}\Big(\frac p\ell-\frac p{ q }\Big)\le \frac{N\alpha(p-1)}{2\beta\ell}
= \alpha<1.
$$
Therefore, the integral on the second term in the right-hand side of~\eqref{eq:inequality.f} is bounded, and we have
\begin{equation*}
\label{eq:inequality.global.bound}
f(t)\le \eta+K_0 f(t)^p
\end{equation*}
with $K_0$ a universal constant.

We now notice that $f(t)\to0$ as $t\to0^+$, since $\|u(\cdot,t)\|_{L^q(\R^N)}$ is bounded in $[0,\tau]$ for any $\tau\in(0,T)$. On the other hand, $f$ is nondecreasing. Hence, if $K_02^p\eta^{p-1}<1$ and $t^b\|\mathcal{L}_{u_0}(\cdot,t)\|_{L^{ q }(\R^N)}\le \eta$ for every $t>0$, the function $f(t)$ can never reach the value $2\eta$. Indeed, if $t_1$ is such that $f(t_1)=2\eta$, then
\[
f(t_1)\le \eta+K_0(2\eta)^{p-1}f(t_1)=\frac12f(t_1)\big(1+K_02^p\eta^{p-1}\big)<f(t_1)\]
a contradiction, and  the theorem is proved.
\end{proof}

\begin{rem}
The global $L^q$-solutions provided by Lemma~\ref{lem:key.lemma} satisfy $\|u(\cdot,t)\|_{L^q(\mathbb{R}^N)}\le 2\eta t^{-b}$. Hence they decay to 0 in $L^q(\mathbb{R}^N)$ as $t\to\infty$.
\end{rem}

\subsubsection{A strong condition for global existence}

Our first condition ensuring~\eqref{eq:eta.condition}, and hence global existence, is of the same type as the one obtained in~\cite{Weissler-1981} for the local case $\alpha=1$, $\beta=1$.
\begin{teo}\label{teo-global existence 2} Let $p\in[p_f,\infty)$ and $q$ satisfying~\eqref{eq:distinguished.range}. There exists a value $\eta_0>0$ depending only on $p,q,\alpha,\beta$ and $N$ such that if $u_0\in L^\ell(\R^N)\cap L^q(\R^N)$ and
\begin{equation}
\label{eq:strong.condition}
\|u_0\|_{L^\ell(\R^N)}\le\eta_0,
\end{equation}
then the maximal $L^q$-solution to~\eqref{eq-problem} with initial datum $u_0$ is global in time.
\end{teo}
\begin{proof}
Let $r$ such that $1-\frac1r=\frac1\ell-\frac1q$. The hypotheses on $q$ imply that
$0<1-\frac1r<\frac{2\beta}N$,
or equivalently, $r\in(1,p_{\rm c})$. Hence, we may use the bound~\eqref{eq-estimate Z} to obtain
\begin{align*}
t^b\|\mathcal{L}_{u_0}(\cdot,t)\|_{L^q(\R^N)}&\le t^b\|u_0\|_{L^\ell(\R^N)}\|Z(\cdot,t)\|_{L^r(\R^N)}
\le Kt^b\|u_0\|_{L^\ell(\R^N)}t^{-\frac{N\alpha}{2\beta}\big(1-\frac1r\big)}\\
&=
Kt^b\|u_0\|_{L^\ell(\R^N)}t^{-\frac{N\alpha}{2\beta}\big(1-\frac1r\big)}\le K\eta_0.
\end{align*}
The result now follows from Lemma~\ref{lem:key.lemma}, taking $\eta_0=\eta/K$.
\end{proof}

\subsubsection{A weaker (\lq\lq sharp'') condition for global existence}

Our second condition on the initial datum ensuring~\eqref{eq:eta.condition}, and hence global existence, is of a different, weaker, nature. We will obtain in Section~\ref{sect-blow up} a condition for blow-up in the same spirit, which shows that condition~\eqref{eq:eta.condition} is \lq\lq sharp''. We may  assume without loss of generality that $p>p_f$, since when $p=p_f$ both conditions for global existence are in fact the same.

\begin{teo}\label{teo-global existence}
Let $p\in(p_f,\infty)$, $q$ satisfying~\eqref{eq:distinguished.range}, and $u_0\in L^\ell(\R^N)\cap L^q(\R^N)$.  Let $\eta>0$ as in Lemma~\ref{lem:key.lemma} and $R_0=(\eta/\|u_0\|_{L^q(\R^N)})^{1/b}$. There is a constant $\delta>0$ depending only on $\eta$  such that if $u_0\ge 0$ and
\begin{equation}\label{eq-condition}
    \sup_{R>R_0}R^{\frac\alpha{p-1}-\frac{N\alpha}{2\beta}}\int_{|y|<R^{\frac\alpha{2\beta}}}u_0(y)\,{\rm d}y<\delta,
\end{equation}
then the maximal $L^q$-solution to~\eqref{eq-problem} exists globally in time. Moreover, it exists globally in $L^\ell(\R^N)$.
\end{teo}
\begin{proof}
Since $\|Z(\cdot,t)\|_{L^1(\mathbb{R}^N}=1$, then $\|\mathcal{L}_{u_0}(\cdot,t)\|_{L^q(\R^N)}\le \|u_0\|_{L^q(\R^N)}$. As $b>0$ when $q>\ell$, then
$$
t^b\|\mathcal{L}_{u_0}(\cdot,t)\|_{L^{ q }(\R^N)}\le R_0^b\|u_0\|_{L^q(\R^N)}=\eta\quad \text{for }t\in(0,R_0].
$$
It remains then to show that~\eqref{eq:eta.condition} holds for $t\ge R_0$.

Let $\Lambda>0$ to be chosen later. We make the decomposition $\mathcal{L}_{u_0}=\mathcal{L}_1+\mathcal{L}_2$, where
$$
\mathcal{L}_1(x,t)=\int_{|y|>\Lambda t^{\frac\alpha{2\beta}}}u_0(y)Z(x-y,t)\,{\rm d}y,\qquad
\mathcal{L}_2(x,t)=\int_{|y|<\Lambda t^{\frac\alpha{2\beta}}}u_0(y)Z(x-y,t)\,{\rm d}y.
$$
Let $r$ be such that $1-\frac1r=\frac1\ell-\frac1q$. The hypotheses on $q$ imply that
\begin{equation*}
\label{eq:condicion.q}
0\le1-\frac1r=\frac{2\beta}{N(p-1)}-\frac1q<\frac{2\beta}N,
\end{equation*}
or equivalently, $r\in[1,p_{\rm c})$. Hence, we may use the bound~\eqref{eq-estimate Z} to obtain, for all $t\ge R_0>0$,
\[\begin{aligned}
	\|\mathcal{L}_1(\cdot,t)\|_{L^{ q }(\R^N)}&\le \|u_0\|_{L^\ell(|y|>\Lambda t^{\frac\alpha{2\beta}})}\|Z(\cdot,t)\|_{L^r(\R^N)}\le K\|u_0\|_{L^\ell(|y|>\Lambda R_0^{\frac\alpha{2\beta}})}\, t^{-\frac{N\alpha}{2\beta}(1-\frac1r)}\le \frac\eta2 t^{-b}
\end{aligned}\]	
if $\Lambda\ge1$ is large enough, since $\frac{N\alpha}{2\beta}(1-\frac1r)=b$.

In order to estimate $\mathcal{L}_2$ we make the decomposition $\mathcal{L}_2=\mathcal{L}_{21}+\mathcal{L}_{22}$, where
\begin{align*}
	\mathcal{L}_{21}(x,t)&=\int_{\scriptsize\begin{array}{c} |y|<\Lambda t^{\frac\alpha{2\beta}}\\
|x-y|<\lambda t^{\frac\alpha{2\beta}}\end{array}}u_0(y)Z(x-y,t)\,{\rm d}y,\\[6pt]
\mathcal{L}_{22}(x,t)&=\int_{\scriptsize\begin{array}{c} |y|<\Lambda t^{\frac\alpha{2\beta}}\\
|x-y|>\lambda t^{\frac\alpha{2\beta}}\end{array}}u_0(y)Z(x-y,t)\,{\rm d}y,
\end{align*}
for some $\lambda \in(0,1)$ small to be further specified later. We take $r$ as above and use the bound~\eqref{eq:pointwise.estimate.Z} to obtain
    \[\begin{aligned}
        \|\mathcal{L}_{21}(\cdot,t)\|_{L^q(\R^N)}&\le  \|u_0\|_{L^\ell(|y|<\Lambda t^{\frac\alpha{2\beta}})}\|Z(\cdot,t)\|_{L^r(|z|<\lambda t^{\frac\alpha{2\beta}})}\le \|u_0\|_{L^\ell(\R^N)}\tilde K \lambda^{2\beta-N(1-\frac1r)}t^{-\frac{N\alpha}{2\beta}(1-\frac1r)}\le \frac{\eta}4t^{-b}
    \end{aligned}\]
if $\lambda$ is small enough, how small depending on $\|u_0\|_{L^\ell(\R^N)}$,
because $2\beta-N(1-\frac1r)>0$.

From now on we fix $\lambda$. We have
\[
	\|\mathcal{L}_{22}(\cdot,t)\|_{L^q(\R^N)}\le \int_{|y|<\Lambda t^{\frac\alpha{2\beta}}}u_0(y)\,{\rm d}y \,\|Z(\cdot,t)\|_{L^{ q }(  |z|\ge \lambda t^{\frac\alpha{2\beta}})}.
\]
The bound~\eqref{eq:pointwise.estimate.Z} for the kernel $Z$ yields
$$
\|Z(\cdot,t)\|_{L^{ q }( |z|\ge \lambda t^{\frac\alpha{2\beta}})}<\widehat K \lambda^{-2\beta-N\big(1-\frac 1q\big)}t^{-\frac{N\alpha}{2\beta}\big(1-\frac1q\big)}
$$
for all $t>0$ for some constant $\widehat K$ depending only on $\lambda$, $\alpha$, $\beta$ and $N$. Therefore, for all $t\ge R_0$,
\[\begin{aligned}
t^b \|\mathcal{L}_{22}(\cdot,t)\|_{L^q(\R^N)}
&\le  \widehat K \lambda^{-2\beta-N\big(1-\frac 1q\big)}t^{\frac\alpha{p-1}-\frac{N\alpha}{2\beta}}\int_{|y|<\Lambda t^{\frac\alpha{2\beta}}}u_0(y)\,{\rm d}y\\
&=  \widehat K \lambda^{-2\beta-N\big(1-\frac 1q\big)}\Lambda^{\frac{2\beta}{p-1}-N}(\Lambda^{\frac{2\beta}\alpha}t)^{\frac\alpha{p-1}-\frac{N\alpha}{2\beta}}\int_{|y|<\Lambda t^{\frac\alpha{2\beta}}}u_0(y)\,{\rm d}y\\
&\le  \widehat K \lambda^{-2\beta-N\big(1-\frac 1q\big)}\Lambda^{\frac{2\beta}{p-1}-N}\sup_{R>\Lambda^{\frac{2\beta}\alpha}R_0}R^{\frac\alpha{p-1}-\frac{N\alpha}{2\beta}}\int_{|y|<R^{\frac\alpha{2\beta}}}u_0(y)\,{\rm d}y\\
&\le \sup_{R>R_0}R^{\frac\alpha{p-1}-\frac{N\alpha}{2\beta}}\int_{|y|<R^{\frac\alpha{2\beta}}}u_0(y)\,{\rm d}y
\end{aligned} \]
if we choose $\Lambda\ge1$ so large that $\widehat K \lambda^{-2\beta-N\big(1-\frac 1q\big)}\Lambda^{\frac{2\beta}{p-1}-N}\le 1$, as $\frac{2\beta}{p-1}-N<0$.
Hence,
\[t^b \|\mathcal{L}_{22}(\cdot,t)\|_{L^q(\R^N)}\le \frac\eta4\]
if~\eqref{eq-condition} holds  with $\delta=\frac\eta4$.

Global existence in $L^q$ now follows from Lemma~\ref{lem:key.lemma}, and then global existence in $L^\ell$ comes as a corollary, thanks to Theorem~\ref{thm:local.p.superfujita}.
\end{proof}

\subsection{Global existence outside the distinguished range}

We will now use the result for $q$ satisfying~\eqref{eq:distinguished.range}  to prove the existence of global $L^q$-solutions, whenever $p\ge p_f$, outside the distinguished range. Let us recall that $p\ge p_f$ if and only if $\ell\ge 1$.

\subsubsection{Outside the good range}

If $1\le \ell <q<p$,  then $p\in[p_f,p_{\rm c})$ and $q\in(\ell,\ell p]$. The existence of a global $L^q$-solution is therefore an immediate corollary of Theorem~\ref{thm:local.p.subcritical} together with either Theorem~\ref{teo-global existence 2} or Theorem~\ref{teo-global existence}.

\begin{coro}
Let $p\in[p_f,p_{\rm c})$, $q\in(\ell,p)$, and $u_0\in L^\ell(\R^N)\cap L^p(\R^N)$. If either~\eqref{eq:strong.condition} or~\eqref{eq-condition} hold, then the global $L^p$-solution to problem~\eqref{eq-problem} with initial datum $u_0$ is a global in time $L^q$-solution, which moreover belongs to $C([0,\infty);L^q(\mathbb{R}^N))$. 	
\end{coro}

If $q\in[1,\ell]$, the existence of a global $L^q$-solution, if $p\in[p_f,\infty)$, is an immediate corollary of Theorem~\ref{thm:local.p.superfujita} together with either Theorem~\ref{teo-global existence 2} or Theorem~\ref{teo-global existence}.

\begin{coro}
Let $p\in[p_f,\infty)$,  $q\in[1,\ell]$, $\hat q$ satisfying~\eqref{eq:distinguished.range}, and $u_0\in L^q(\mathbb{R}^d)\cap L^{\hat q}(\mathbb{R}^d)$. If either~\eqref{eq:strong.condition} or~\eqref{eq-condition} hold, then the global $L^{\hat q}$-solution to problem~\eqref{eq-problem} with initial datum $u_0$ is a global in time $L^q$-solution, which moreover belongs to $C([0,\infty);L^q(\mathbb{R}^N))$.
\end{coro}

\subsubsection{Outside the distinguished range, but yet in the good range}
We yet have to show the existence of global solutions for $q\in(\ell p,\infty]$, with $p\in [p_f,\infty)$.  This will be a corollary of the following result.
\begin{teo}
Let $p\in [p_f,\infty)$,  $q\in(\ell p,\infty]$, and $u_0\in L^\ell(\R^N)\cap L^q(\R^N)$. Let $u$ be an $L^{\ell p}$-solution to~\eqref{eq-problem} with initial datum $u_0$ and existence time $T$. Then $u$ is also an $L^q$-solution with initial datum $u_0$ and existence time $T$. Moreover, $u\in C([0,T);L^q(\mathbb{R}^N))$.
\end{teo}

\begin{proof}
We start by introducing some auxiliary functions and numbers. For any $j\in\mathbb{N}\cup\{0\}$ and $p\in [1,\infty)$ we define $f_j(p)=1-p^j(p-1)$. For any fixed $j$, the function $f_j$ is strictly decreasing in $p$, and has a unique root, that we denote by $p_j$, so that $f_j(p)>0$ for $p\in[1,p_j)$ and $f_j(p)<0$ for $p>p_j$. For instance, $p_0=2$, $p_1=(1+\sqrt5)/2$. On the other hand, for any fixed $p\in (1,\infty)$ the function $j\to f_j(p)$ is strictly decreasing in $j$ and goes to $-\infty$ as $j\to\infty$. Therefore, $p_j\searrow 1$ as $j\to\infty$, so that $p_j<p_f$ if $j$ is large enough. While $p_j>p_f$, which occurs at least for $j=0,1$, we define $q_j:[p_f,p_j)\to\mathbb{R}_+$ by $q_j(p)=\ell/f_j(p)$. An easy analysis of derivatives shows that $pf_j(p)\in (0,1)$ if $p\in[p_f,p_j)$, so that $\ell p< q_j(p)$ in that interval.

Let $r_0$ such that $1-\frac1{r_0}=\frac1\ell-\frac1 q$.   Since $q>\ell p\ge \ell p_f>\ell$, then $1-\frac1{r_0}>0$. Besides,
$$
\frac1\ell-\frac1 q=1-\frac1{r_0}<\frac{2\beta}N=\frac{p-1}\ell
$$
if and only if $\frac1q>\frac{f_0(p)}\ell=\frac1{q_0(p)}$. This holds, when $q>\ell p$, if and only if
$$
 q\in(\ell p,\infty]\text{ if } p>p_0, \quad   q\in(\ell p,\infty) \text{ if } p=p_0, \quad   q\in(\ell p,q_0(p)) \text{ if } p\in [p_f,p_0).
$$
Under these restrictions on $q$, if $0<t<\tau<T$, then, using~\eqref{eq-estimate Y},
$$
\|\mathcal{N}_u(\cdot,t)\|_{L^ q (\R^N)}\le \int_0^t\||u(\cdot,s)|^p\|_{L^\ell(\R^N)}\|Y(\cdot,t-s)\|_{L^{r_0}(\R^N)}\,{\rm d}s\le
K\sup_{s\in(0,\tau)}\|u(\cdot,s)\|_{L^{\ell p}(\R^N)}^p \frac{\tau^{1-\sigma(r_0)}}{1-\sigma(r_0)},
$$	
whence the result.

To cover the case $p=p_0=2$, $q=\infty$, we take $\tilde r_0$ such that $1-\frac1{\tilde r_0}=\frac1{2\ell}=\frac{\beta}N\in (0,\frac{2\beta}N)$. Using that $u$ is an  $L^{2 \ell p_0}$-solution (see the previous step) and also estimate~\eqref{eq-estimate Y}, for $0<t<\tau< T$ we have
$$
\|\mathcal{N}_u(\cdot,t)\|_{L^\infty(\R^N)}\le K\sup_{s\in(0,\tau)}\|u(\cdot,s)\|_{L^{2\ell p_0}(\R^N)}^{p_0} \frac{\tau^{1-\sigma(\tilde r_0)}}{1-\sigma(\tilde r_0)}.
$$

We next consider the case $p\in [p_f,p_0)$, $q=q_0(p)$. Let $\varepsilon\in(0,\frac{q_0(p)}{\ell p}-1)$, so that $\frac{q_0(p)}{1+\varepsilon}\in (\ell p,q_0(p))$. Since $u_0\in L^{\ell}(\R^N)\cap L^{q_0(p)}(\R^N)$, then $u_0\in L^{\frac{q_0(p)}{1+\varepsilon}}(\R^N)$, and by a previous step we already know that $u$ is an  $L^{\frac{q_0(p)}{1+\varepsilon}}$-solution.  Let $\bar r_0$ such that $1-\frac1{\bar r_0}=\frac{p(1+\ep)}{q_0(p)}-\frac1{q_0(p)}$. Assuming in addition that
$$
\ep\in\Big(0,\frac{p-1}p\big(\frac{q_0(p)}\ell-1\big)\Big),
$$
then $1-\frac1{\bar r_0}\in (0,\frac{2\beta}N)$. Therefore, using once more~\eqref{eq-estimate Y}, for $0<t<\tau< T$ we  deduce that
$$
\|\mathcal{N}_u(\cdot,t)\|_{L^{q_0(p)}(\R^N)}\le K\sup_{s\in(0,\tau)}\|u(\cdot,s)\|_{L^{\frac{q_0(p)}{1+\varepsilon}}(\R^N)}^p \frac{\tau^{1-\sigma(\bar r_0)}}{1-\sigma(\bar r_0)}.
$$

We now proceed by iteration, until $p_j\le p_f$, something that will happen after a finite number of steps, starting from the already studied case $j=0$. So we assume that $p_{j-1}>p_f$, and we have yet to cover the case $p\in[p_f,p_{j-1})$, $q>q_{j-1}(p)$.

Let $r_j$ such that $1-\frac1{r_j}=\frac p{q_{j-1}(p)}-\frac1 q$.   Since $q>q_{j-1}(p)$, then $1-\frac1{r_0}>0$. Besides,
$$
\frac p{q_{j-1}(p)}-\frac1 q=1-\frac1{r_j}<\frac{2\beta}N=\frac{p-1}\ell
$$
if and only if $\frac1q>\frac p{q_{j-1}(p)}-\frac{p-1}\ell=\frac{pf_{j-1}(p)+1-p}\ell=\frac{f_j(p)}\ell=\frac1{q_j(p)}$. This holds, when $q>\ell p$, if and only if
$$
 q\in(\ell p,\infty]\text{ if } p>p_j, \quad   q\in(\ell p,\infty) \text{ if } p=p_j, \quad   q\in(\ell p,q_j(p)) \text{ if } p\in [p_f,p_j),
$$
and these cases are covered reasoning as above.

To deal with the case $p=p_j$, $q=\infty$, we take $\tilde r_j$ such that $1-\frac1{\tilde r_j}=\frac1{2\ell}=\frac{\beta}N\in (0,\frac{2\beta}N)$. Using that $u$ is an  $L^{2 \ell p_j}$-solution, something that we have just proved, we conclude, reasoning as in the case $j=0$, that $u$ is an $L^\infty$-solution.

To complete the argument we have to consider the case $p\in [p_f,p_j)$, $q=q_j(p)$. We take
$$
\varepsilon\in\Big(0,\min\Big\{\big(\frac{q_j(p)}{\ell p}-1\big),\frac{p-1}p\big(\frac{q_j(p)}\ell-1\big)\Big\}\Big),
$$
and $\bar r_j$ such that $1-\frac1{\bar r_j}=\frac{p(1+\ep)}{q_j(p)}-\frac1{q_j(p)}$ and reason as in the case $j=0$ to conclude that $j$ is an $L^{q_j(p)}$-solution.

We omit the proof of continuity, since it is similar to that of Proposition~\ref{prop:continuity}.
\end{proof}

\begin{coro}
Let $p\in [p_f,\infty)$,  $q\in(\ell p,\infty]$, and $u_0\in L^\ell(\R^N)\cap L^q(\R^N)$. If either~\eqref{eq:strong.condition} or~\eqref{eq-condition} hold, then the global $L^{\ell p}$-solution to problem~\eqref{eq-problem} with initial datum $u_0$ is a global in time $L^q$-solution, which moreover belongs to $C([0,\infty);L^q(\mathbb{R}^N))$. 	
\end{coro}

\section{Blow-up results}\label{sect-blow up}
 \setcounter{equation}{0}

We devote this section to the proof of a blow-up result which shows, in particular, that:
\begin{itemize}
\item for any $p\in(1,\infty)$ and $q\in[1,\infty]$ there are $L^q$-solutions that blow up in finite time;
\item all nontrivial nonnegative $L^q$-solutions of~\eqref{eq-problem} blow up if $p\in(1,p_f)$, which combined with the results of the previous section shows that $p_f$ is the Fujita exponent for this problem.
\item condition~\eqref{eq-condition} for global existence is \lq\lq sharp''.
\end{itemize}

\begin{teo}\label{nonglobal}
Let  $p\in(1,\infty)$ and  $q\in[1,\infty]$. There is a  constant $C= C(N, \alpha,\beta,p)$ such that, if $u$ is a nontrivial nonnegative $L^q$-solution to~\eqref{eq-problem} with initial datum in $L^q(\mathbb{R}^N)$ satisfying
\begin{equation}
\label{eq:condition.nonglobal}
R^{\frac{\alpha}{p-1}-\frac{N\alpha}{2\beta}}\int_{|y|<R^{\frac\alpha{2\beta}}} u_0(y)\,{\rm d}y >C\quad\text{for some }R>0,
\end{equation}
then $\limsup\limits_{t\nearrow T}\|u(\cdot,t)\|_{L^q(\R^N)}=\infty$ for some $T\in(0,\infty)$.
\end{teo}

Before proceeding to the proof of the theorem, let us give two important corollaries.

\begin{coro}
Let $p\in(1,\infty)$ and $q\in[1,\infty]$. There are $L^q$-solutions to~\eqref{eq-problem} that blow up in finite time.
\end{coro}
\begin{proof}
Let $M>0$ and $u_0=M\chi_{B_1(0)}$. Since $u_0\in L^1(\mathbb{R}^N)\cap L^\infty(\mathbb{R}^N)$,  there are $L^q$-solutions to~\eqref{eq-problem} with this initial datum for all $q\in[1,\infty]$; see theorems~\ref{teo-local existence}, \ref{thm:local.p.subcritical}, and~\ref{thm:local.p.superfujita}. For this $u_0$, any fixed value $R>1$, and $C$ as in Theorem~\ref{nonglobal},
$$
R^{\frac{\alpha}{p-1}-\frac{N\alpha}{2\beta}}\int_{|y|<R^{\frac\alpha{2\beta}}} u_0(y)\,{\rm d}y=M\omega_NR^{\frac{\alpha}{p-1}-\frac{N\alpha}{2\beta}} >C
$$
if $M$ is large enough. Hence, any $L^q$-solution with this initial datum blows up in finite time.
\end{proof}

\begin{coro}\label{pmenor}
Let $p\in(1,p_f)$ and $q\in[1,\infty]$. Let $u$ be a nontrivial nonnegative $L^q$-solution to~\eqref{eq-problem}. There exists $T\in(0,\infty)$ such that $\limsup\limits_{t\nearrow T}\|u(\cdot,t)\|_{L^q(\R^N)}=\infty$.
\end{coro}
\begin{proof}
Let $\bar R>0$ such that $\int_{|y|<\bar R^{\frac\alpha{2\beta}}} u_0(y)\,{\rm d}y  =B>0$ and $C>0$ as in Theorem~\ref{nonglobal}. Since $p<p_f$, then  ${\frac{\alpha}{p-1}-\frac{N\alpha}{2\beta}}>0$, and~\eqref{eq:condition.nonglobal} holds for any $R>\bar R$ such that $R^{\frac{\alpha}{p-1}-\frac{N\alpha}{2\beta}}B>C$.
\end{proof}

From now on we always assume, without further mention, that $u$ is an nontrivial nonnegative $L^q$-solution to~\eqref{eq-problem} with an initial datum $u_0\in L^q(\mathbb{R}^N)$ for some $q\in[1,\infty]$.

The proof of Theorem~\ref{nonglobal}, which is organized in three lemmas, goes as follows. Let $G$ be the profile of the kernel $Y$; see~\eqref{eq:ss.structure}. For any $b>0$ and $T>0$ we define
\begin{gather*}
J(T;b)= \int_0^{T}\int_{\R^N} u(x,t) Y_b(x,T-t)\,{\rm d}x{\rm d}t,\quad\text{where}
\\
Y_b(x,T-t)=\chi_{\{|x|<(T-t)^{\frac{\alpha}{2\beta}}\}} (T-t)^{\frac{p\alpha}{p-1} -1}\bar G(b)\quad\text{and }
\bar G(b):=\min_{\{|\xi|<b \}} G(\xi).
\end{gather*}

We first prove that $J$ is finite as long as  $\|u(\cdot,t)\|_{L^q(\R^N)}$ remains bounded.

\begin{prop}\label{prop-existencia}
Let $p>1$ and $b>0$. If $\|u(\cdot,t)\|_{L^q(\R^N)}\le C$ for $t\in[0,T)$, then $J(T;b)$ is finite.
\end{prop}	
\begin{proof}
Let $r$ be such that $1/r+1/ q =1$. Notice that $r\in[1,\infty]$. Then
$$
J(T;b)\leq \int_0^{T} \|u(\cdot,t)\|_ q \|Y_b(\cdot,T-t)\|_r\,{\rm d}t.
$$
Since $\|Y_b(\cdot,T-t)\|_{L^r(\R^N)}=\bar G(b)C(N,r)(T-t)^{\frac{N\alpha}{2\beta r}}(T-t)^{\frac {p\alpha}{p-1}-1}$, then $J(T;b)<\infty$.
\end{proof}
The next step is to obtain an estimate from above for $J(T;b)$, whenever this quantity is finite, independent of the initial data.

\begin{prop}\label{prop-des2} If $b\ge 2^{\frac \alpha{2\beta}+1}$, there exists $C_1=C_1(N,\alpha,\beta,p,b)>0$ such that,  if $J(T;b)<\infty$, then
\begin{equation}
\label{eq:estimate.J.above}
J(T;b)< C_1 T^{\alpha+\frac{N\alpha}{2\beta}}.
\end{equation}
\end{prop}	
\begin{proof}
Since $u$ is a solution of the integral equation~\eqref{eq-formula},
\[\begin{aligned}
J(T;b)&> \int_0^{T}\int_{\R^N} \int_0^{t}\int_{\R^N}u^p(y,s)Y(x-y,t-s)\,{\rm d}y{\rm d}s\,Y_b(x,T-t)\,{\rm d}x{\rm d}t\\
&= \int_0^{T}\int_{\R^N} u^p(y,s)  \int_{s}^{T}\int_{\R^N}Y(x-y,t-s) Y_b(x,T-t)\,{\rm d}x{\rm d}t\,{\rm d}y{\rm d}s.
\end{aligned}\]
If  $s<T$, $t>s+(\frac{2}{b})^{\frac{2\beta}\alpha}(T-s)$, $|x|<(T-t)^{\frac\alpha{2\beta}}$ and $|y|< (T-s)^{\frac\alpha{2\beta}}$, then
\[|x-y|<2(T-s)^{\frac\alpha{2\beta}}<b(t-s)^{\frac\alpha{2\beta}}.\]
Hence, since $\alpha-1-\frac{N\alpha}{2\beta}<0$, if $b\ge 2^{\frac \alpha{2\beta}+1}$ we get, using the selfsimilar structure~\eqref{eq:ss.structure} of $Y$,
\begin{align*}
J(T;b)&>\bar G(b)\! \int_0^{T}\int_{|y|<(T-s)^{\frac\alpha{2\beta}}}u^p(y,s)\!\int_{s+(\frac{2}{b})^{\frac{2\beta}\alpha}(T-s)}^T
\!\int_{|x|<(T-t)^{\frac\alpha{2\beta}}}Y(x-y,t-s)(T-t)^{\frac{p\alpha}{p-1} -1}\,{\rm d}x{\rm d}t\,{\rm d}y{\rm d}s\\
	&\ge \bar G(b)^2\!\int_0^{T}\int_{|y|<(T-s)^{\frac\alpha{2\beta}}}u^p(y,s) \!\int_{s+(\frac{2}{b})^{\frac{2\beta}\alpha}(T-s)}^T\!\int_{|x|<(T-t)^{\frac\alpha{2\beta}}}(t-s)^{\alpha-1-\frac{N\alpha}{2\beta}}
(T-t)^{\frac{p\alpha}{p-1}-1}\,{\rm d}x{\rm d}t\,{\rm d}y{\rm d}s\\
&\ge C(N)\bar G(b)^2\int_0^{T}\int_{|y|<(T-s)^{\frac\alpha{2\beta}}}u^p(y,s)(T-s)^{\alpha-1-\frac{N\alpha}{2\beta}}
\int_{s+(\frac{2}{b})^{\frac{2\beta}\alpha}(T-s)}^T(T-t)^{\frac{p\alpha}{p-1}-1+\frac{N\alpha}{2\beta}}\,{\rm d}t\,{\rm d}y{\rm d}s.
\end{align*}
 We conclude that
\begin{equation*}
\label{eq:prelim.est.J.below}
\begin{gathered}
J(T;b)>C(N,\alpha,\beta,p,b)\int_{0}^{T}\int_{\R^N} u^p(y,s)  \hat Y_b(y,T-s) (T-s)^{\frac{p\alpha}{p-1}}\,{\rm d}y{\rm d}s,\quad \text{where}\\
\hat Y_b(x,T-t)= \chi_{\{|x|<(T-t)^{\frac{\alpha}{2\beta}}\}} (T-t)^{\alpha -1}\bar G(b)
\end{gathered}
\end{equation*}
for some constant $C(N,\alpha,\beta,p) >0$. Using Jensen's inequality,
\begin{align*}
J(T;b)&> C(N,\alpha,\beta,p,b)\int_{0}^{T}\int_{\R^N} u^p(y,s)  \hat Y_b(y,T-s) (T-s)^{\frac{p\alpha}{p-1}}\,{\rm d}y{\rm d}s \\
&\ge C(N,\alpha,\beta,p,b)\Big(\int_0^{T}\int_{\R^N}u(y,s)\frac{\hat Y_b(y,T-s)}{\|\hat Y_b\|_{L^1(\mathbb{R}^N\times (0,T))}} (T-s)^{\frac{\alpha}{p-1}}\,{\rm d}y{\rm d}s\Big)^p\|\hat Y_b\|_{L^1(\mathbb{R}^N\times (0,T))}\\
&=C(N,\alpha,\beta,p,b) \Big(\int_0^{T} \int_{\R^N} u(y,s) Y_b(y,T-s)\,{\rm d}y{\rm d}s\Big )^p  \frac{1}{\|\hat Y_b\|^{p-1}}.
\end{align*}
A direct computation shows that
\[\begin{aligned}
     \|\hat Y_b\|_{L^1(\mathbb{R}^N\times (0,T))}&=\bar G(b)\int_0^T\int_{|x|<(T-t)^{\frac{\alpha}{2\beta}}}(T-t)^{\alpha-1}\,{\rm d}x{\rm d}t
     =C(N)\bar G(b)\int_0^T(T-t)^{\alpha-1+\frac{N\alpha}{2\beta}}\,{\rm d}t\\
     &=C(N,\alpha,\beta,p)\bar G(b) T^{\alpha+\frac{N\alpha}{2\beta}}.
\end{aligned}\]
Hence, $J(T;b)>\bar C (N,\alpha,\beta,p,b) J(T;b)^p  T^{-({\alpha }+\frac{N\alpha}{2\beta})(p-1)}$, whence \eqref{eq:estimate.J.above}.
\end{proof}

Finally, we obtain an estimate from below for $J(T;b)$ that shows that this quantity is big for certain initial data.
\begin{prop}\label{noexist2}
There exists
 $C_2=C_2(N, \alpha,\beta,p,b)>0$  such that
 $$
 J(T;b)\ge C_2 A(T)T^{\alpha +\frac{N\alpha}{2\beta}},\quad\text{where } A(T):= T^{\frac{\alpha}{p-1}-\frac{N\alpha}{2\beta}}\int_{|y|<(T/4)^{\frac\alpha{2\beta}}} u_0(y)\,{\rm d}y.
 $$
\end{prop}
\begin{proof}
If  $|x|<t^{\frac\alpha{2\beta}}$  and  $|y|< t^{\frac\alpha{2\beta}}$, then $|x-y|<2t^{\frac\alpha{2\beta}}$. Hence,  if $|x|<t^{\frac\alpha{2\beta}}$,
\[\begin{aligned}u(x,t) &> \int_{\R^N} u_0 (y) Z(x-y,t)\,{\rm d}y\ge \int_{|y|< t^{\frac\alpha{2\beta}} } u_0 (y) Z(x-y,t)\,{\rm d}y\ge C(N,\alpha,\beta)  t^{-\frac{N\alpha}{2\beta}}\int_{|y|< t^{\frac\alpha{2\beta}} }u_0(y)\,{\rm d}y,
\end{aligned}\]
since $Z(x-y,t)= t^{-\frac{N\alpha}{2\beta}}G\big((x-y)t^{-\frac\alpha{2\beta}}\big)\ge t^{-\frac{N\alpha}{2\beta}}\bar G(2)$ if $|x-y|<2t^{\frac\alpha{2\beta}}$. Therefore,
$$
J(T;b)\ge C(N,\alpha,\beta)\bar G(b)  \int_{T/4}^{T/2}\int_{|y|< t^{\frac\alpha{2\beta}} }u_0(y)\,{\rm d}y\int_{|x|<{\min\{ t^{\frac\alpha{2\beta}}, (T-t)^{\frac\alpha{2\beta}}\}}}  t^{-\frac{N\alpha}{2\beta}} (T-t)^{\frac{p\alpha}{p-1} -1}\,{\rm d}x{\rm d}t.
$$
But $\min\{ t^{\frac\alpha{2\beta}}, (T-t)^{\frac\alpha{2\beta}}\}=t^{\frac\alpha{2\beta}}\ge (T/4)^{\frac\alpha{2\beta}}$ if $t\in [T/4,T/2]$. Thus,
\begin{align*}
J(T;b)& \ge \bar C(N,\alpha,\beta)\bar G(b) \int_{|y|<(T/4)^{\frac\alpha{2\beta}}} u_0(y)\,{\rm d}y  \int_{T/4}^{T/2}(T-t)^{\frac{p\alpha}{p-1} -1}\, {\rm d}t\\
&\ge C_2(N,\alpha,\beta,p,b)\int_{|y|<(T/4)^{\frac\alpha{2\beta}}} u_0(y)\,{\rm d}y \;T^{\frac{p\alpha}{p-1}}\\
&=C_2(N, \alpha,\beta,p,b)\ T^{\alpha +\frac{N\alpha}{2\beta}}\Big(T^{\frac{\alpha}{p-1}-\frac{N\alpha}{2\beta}}\int_{|y|<(T/4)^{\frac\alpha{2\beta}}} u_0(y)\,{\rm d}y \Big)  = C_2(N, \alpha,\beta,p,b) A(T)  T^{\alpha +\frac{N\alpha}{2\beta}}.\qedhere
\end{align*}
\end{proof}
We have now all the ingredients to prove Theorem~\ref{nonglobal}.
\begin{proof}[Proof of Theorem~\ref{nonglobal}]
We take $b\ge 2^{\frac \alpha{2\beta}+1}$. If $\|u(\cdot,t)\|_{L^q(\R^N)}$ is bounded in $[0,T)$, then $J(T;b)$ is well defined and finite; see Proposition~\ref{prop-existencia}. Hence, $J(T;b)< C_1 T^{\alpha+\frac{N\alpha}{2\beta}}$, with $C_1$ as in Proposition~\ref{prop-des2}. If $A(T)\ge C_1/C_2$, we get a contradiction; see~Proposition~\ref{noexist2}.  This inequality for $A(T)$ will hold true for $T=4R>0$ if~\eqref{eq:condition.nonglobal} holds for some $R>0$ and $C=4^{\frac{N\alpha}{2\beta}-\frac{\alpha}{p-1}}C_1/C_2$.
\end{proof}
\begin{rem}
The proof of Theorem~\ref{nonglobal} gives an upper bound for the existence time when~\eqref{eq:condition.nonglobal} holds. Indeed, $T\le 4R$.
\end{rem}

\section*{Acknowledgments}

This research was supported by the European Union's Horizon 2020 research and innovation programme under the Marie Sklodowska-Curie grant agreement No.\,777822, and by grant CEX2019-000904-S, funded by MCIN/AEI/10.13039/501100011033.

\noindent C. Cort\'azar was also supported by  FONDECYT grant 1190102 (Chile).

\noindent F. Quir\'os was also supported by grants PID2020-116949GB-I00 and RED2022-134784-T, both of them funded by MCIN/AEI/10.13039/501100011033, and by the Madrid Government (Comunidad de Madrid – Spain) under the multiannual Agreement with UAM in the line for the Excellence of the University Research Staff in the context of the V PRICIT (Regional Programme of Research and Technological Innovation).

\noindent N. Wolanski was also supported by  ANPCyT PICT2016-1022.

\section*{Data availability statement}

The manuscript has no associated data.


 \end{document}